\documentclass[12pt]{article}
\usepackage{BOONDOX-cal}
\usepackage{amssymb,a4}
\usepackage{amsmath,amsfonts,amssymb,amsthm, mathrsfs}
\usepackage{tikz}

\allowdisplaybreaks[4]
\newtheorem{thm}{Theorem}[section]
\newtheorem{lem}[thm]{Lemma}
\newtheorem{defi}[thm]{Definition}

\begin{document}
\begin{center}
{\Large \bf Classification of the irreducible ordinary modules for affine vertex operator superalgebras}
\end{center}

\begin{center}
{Huaimin Li, Qing Wang
		\footnote{Supported by China NSF grants No. 12571033.}\\
School of Mathematical Sciences, Xiamen University, Xiamen, China 361005}	
\end{center}

\begin{abstract}
Let $\mathfrak{g}$ be a basic classical Lie superalgebra,
$\mathcal{k}=\frac{h^{\vee}}{u}-h^{\vee}$ a boundary admissible level of $\widehat{\mathfrak{g}}$,
where $u$ is a positive integer and
$h^{\vee}$ is the dual Coxeter number of $\mathfrak{g}$.
In this paper, we classify the irreducible ordinary modules for the
affine vertex operator superalgebra $L_{\widehat{\mathfrak{g}}}(\mathcal{k},0)$ associated to any basic classical Lie superalgebra $\mathfrak{g}$.
More specifically, if $\mathfrak{g}$ is a basic classical Lie superalgebra of type I,
we prove that $L_{\widehat{\mathfrak{g}}}(\mathcal{k},0)$ has exactly $u$ inequivalent irreducible ordinary modules.
If $\mathfrak{g}$ is a finite dimensional simple Lie algebra or a basic classical Lie superalgebra of type II,
we prove that $L_{\widehat{\mathfrak{g}}}(\mathcal{k},0)$ itself is the only irreducible ordinary $L_{\widehat{\mathfrak{g}}}(\mathcal{k},0)$-module.
\end{abstract}
\section{Introduction}
Admissible representations were first studied by Kac and Wakimoto to explore modular invariant representations and
the modular invariance property of normalized (super)characters of admissible modules for
the affine Lie (super)algebras associated to any simple finite-dimensional Lie algebra or $osp(1|2n)$ has also been proved in \cite{KW88}.
Furthermore, they classified the admissible weights for all affine Lie algebras in \cite{KW89}.
Let $\mathfrak{g}$ be a basic classical Lie superalgebra, $\widehat{\mathfrak{g}}$ be the corresponding affine Lie superalgebra, and
$L_{\widehat{\mathfrak{g}}}(\mathcal{k},0)$ be the simple quotient of the universal affine vertex operator superalgebra of $\widehat{\mathfrak{g}}$
at the level $\mathcal{k}$.
A vertex operator superalgebra $L_{\widehat{\mathfrak{g}}}(\mathcal{k},0)$ is called admissible if it is isomorphic to an admissible representation as a $\widehat{\mathfrak{g}}$-module, and $\mathcal{k}$ is called an admissible level.
In \cite{AM95}, Adamovi\'{c} and Milas proved that for the admissible level $\mathcal{k}$, $\mathfrak{g}=sl_2$,
a $\widehat{\mathfrak{g}}$-module $L(\lambda)$ is a $L_{\widehat{\mathfrak{g}}}(\mathcal{k},0)$-module if and only if
$\lambda$ is an admissible weight of level $\mathcal{k}$,
and they presented a conjecture for arbitrary affine Lie algebras.
This conjecture was proved by Arakawa in \cite{Ar16}, that is,
Arakawa proved that $L_{\widehat{\mathfrak{g}}}(\mathcal{k},0)$ at the admissible level $\mathcal{k}$ is rational in the category $\mathcal{O}$
for any simple finite-dimensional Lie algebra $\mathfrak{g}$ and the representations of $L_{\widehat{\mathfrak{g}}}(\mathcal{k},0)$
are equivalent to admissible representations of $\widehat{\mathfrak{g}}$ of level $\mathcal{k}$ whose integral root system is isometric
to that of $L(\mathcal{k}\Lambda_0)$.
This result was extended to the only anisotropic basic Lie superalgebra $\mathfrak{g}=osp(1|2n)$ by Gorelik and Serganova in \cite{GS22}.
Moreover, an interesting bijection between admissible representations for the affine Lie algebra $\widehat{\mathfrak{g}}$ at boundary admissible levels and $\mathbb{C}^\times$ fixed points in homogeneous elliptic affine Springer fibres for the Langlands dual affine Lie algebra $\widehat{\mathfrak{g}}^\vee$
has been constructed by Shan, Xie and Yan in \cite{SXY24}.
On the other hand,
Creutzig, Huang and Yang proved that the ordinary module category of $L_{\widehat{\mathfrak{g}}}(\mathcal{k},0)$
is a braided tensor category for any simple finite-dimensional Lie algebra $\mathfrak{g}$ at the admissible level $\mathcal{k}$ in \cite{CHY18}.
By using the theory of vertex superalgebra extensions,
Creutzig, Genra and Linshaw classified all the irreducible ordinary modules for $L_{\widehat{osp(1|2n)}}(\mathcal{k},0)$ at the admissible level $\mathcal{k}$
and proved that the ordinary module category of $L_{\widehat{osp(1|2n)}}(\mathcal{k},0)$ is a fusion supercategory in \cite{CGL24}.
While except the simple finite-dimensional Lie algebra and $osp(1|2n)$,
for other basic classical Lie superalgebras,
modular invariance occurs in boundary level admissible modules \cite{GK15,KW17}.
As a representative, in \cite{LW25}, we classified the irreducible weak modules of $L_{\widehat{sl(2|1)}}(\mathcal{k},0)$ in the category $\mathcal{O}$
and furthermore classified the irreducible ordinary modules at a boundary admissible level $-\frac{1}{2}$ and at a non-boundary admissible level $\frac{1}{2}$,
we see that $L_{\widehat{sl(2|1)}}(-\frac{1}{2},0)$ is rational in the category $\mathcal{O}$,
while for the case of non-boundary admissible level $\frac{1}{2}$, we showed that there are infinitely many irreducible weak
$L_{\widehat{sl(2|1)}}(\frac{1}{2},0)$-modules in the category $\mathcal{O}$ and it is not rational in the category $\mathcal{O}$,
and we presented a conjecture that for the boundary admissible level $\mathcal{k}$, $L_{\widehat{\mathfrak{g}}}(\mathcal{k},0)$
is rational in the category $\mathcal{O}$ and the irreducible weak modules in the category $\mathcal{O}$ are exactly the admissible modules of level $\mathcal{k}$
for any basic Lie superalgebra.
Gorelik and Kac proved this for any basic Lie superalgebra except the case $D(2|1;a)$ in \cite{GK25}.
Based on Gorelik-Kac's work, we classify the irreducible ordinary modules of $L_{\mathfrak{g}}(\mathcal{k},0)$ for any basic classical Lie superalgebra $\mathfrak{g}$ at the boundary admissible level $\mathcal{k}$ in this paper.

Let $\mathfrak{g}$ be a basic classical Lie superalgebra except $D(2|1;a)$, $\mathfrak{h}$ a Cartan subalgebra of $\mathfrak{g}$,
we fix simple roots as the distinguished simple roots (cf. \cite{Mu12}).
Let $\mathcal{k}$ be a boundary principal admissible level of $\widehat{\mathfrak{g}}$,
from the classification result for boundary admissible level of indecomposable non-twisted affine Kac-Moody superalgebra \cite{GK25},
we know that if $\mathfrak{g}\ne osp(1|2n)$, all boundary admissible levels of $\widehat{\mathfrak{g}}$ are principal.
Let $L(\hat{\lambda})~(\hat{\lambda}\in \widehat{\mathfrak{h}}^*)$ be an irreducible highest weight $\widehat{\mathfrak{g}}$-module,
we prove that $L(\hat{\lambda})$ is an ordinary $L_{\widehat{\mathfrak{g}}}(\mathcal{k},0)$-module if and only if
$\hat{\lambda}$ is a principal admissible weight of level $\mathcal{k}$ such that $L(\lambda)$
is a finite-dimensional $\mathfrak{g}$-module, where $\lambda=\hat{\lambda}|_{\mathfrak{h}}$.
Together with these two conditions,
we can classify the irreducible ordinary modules for the vertex operator superalgebra $L_{\widehat{\mathfrak{g}}}(\mathcal{k},0)$.
We recall the definition of a basic classical Lie superalgebra $\mathfrak{g}$ being of type I if $Q/Q^0\cong\mathbb{Z}$
and $\mathfrak{g}$ being of type II if $Q/Q^0\cong\mathbb{Z}_2$, where $Q$ is the root lattice of $\mathfrak{g}$, $Q^0$ is its even root lattice (cf. \cite{Kac77,S11}),
and basic classical Lie superalgebras of type I consist of type $A(n,m)$ and type $C(n+1)$, while
basic classical Lie superalgebras of type II consist of type $B(m,n)$, type $D(m,n)$, type $F(4)$, type $G(3)$ and type $D(2|1;a)$.
For basic classical Lie superalgebras of type I,
from the definition by $Q/Q^0\cong\mathbb{Z}$, we see that the common point of Lie superalgebras of type I
is that the even roots are independent with odd simple root( it means that the even root can only be written as the linear combinations of the even simple roots), then the condition for a $\mathfrak{g}$-module being finite-dimensional is determined by only even simple roots, in this case,
we deduce that $L_{\widehat{\mathfrak{g}}}(\mathcal{k},0)$ has exactly $u$ inequivalent irreducible ordinary modules.
For basic classical Lie superalgebras of type II,
since we need normalize nondegenerate even supersymmetric invariant bilinear form by $h^\vee\geq0$ and $(\alpha,\alpha)=2$ for a positive root of maximal square length,
note that we also have $h^\vee=(\rho,\theta)+\frac{1}{2}(\theta,\theta)$, where $\rho$ is the Weyl vector and $\theta$ is the highest root,
thus we need divide the proof for type $B(m,n)$ and $D(m,n)$
into two cases with $m>n$ and $m\leq n$.
Together with $F(4)$ and $G(3)$,
we prove that $L_{\widehat{\mathfrak{g}}}(\mathcal{k},0)$ itself is the only irreducible ordinary $L_{\widehat{\mathfrak{g}}}(\mathcal{k},0)$-module.

This paper is organized as follows.
In Section \ref{sec:2},
we recall some facts about affine Lie superalgebra and admissible weights.
In Section \ref{sec:3},
after presenting the concepts and facts about $\mathbb{Z}$-graded vertex operator superalgebra,
we give a criterion for determining that $L(\hat{\lambda})$ is an irreducible ordinary $L_{\widehat{\mathfrak{g}}}(\mathcal{k},0)$-module.
In Section \ref{sec:4},
we classify the irreducible ordinary modules of $L_{\mathfrak{g}}(\mathcal{k},0)$ for any basic classical Lie superalgebra $\mathfrak{g}$ except $D(2|1;a)$ at the boundary principal admissible level $\mathcal{k}$.

\section{Affine Lie superalgebra}
\label{sec:2}
	\def\theequation{2.\arabic{equation}}
	\setcounter{equation}{0}

Let $\mathfrak{g}$ be a basic classical Lie superalgebra, $\mathfrak{h}$ a Cartan subalgebra of $\mathfrak{g}$.
Denote by $\Delta$ (resp. $\Delta^0,\Delta^1$) the root system (resp. the set of even root, the set of odd root),
$Q=\mathbb{Z}\Delta$ the root lattice and $W$ the Weyl group.
We fix the set of simple roots $\Pi=\{\alpha_1,\cdots,\alpha_n\}$, and let $\theta$ be the highest root,
$Q_+$ (resp. $Q_+^0$) the positive root lattice (resp. positive even root lattice),
$\mathfrak{g}=\mathfrak{n}_-\oplus\mathfrak{h}\oplus\mathfrak{n}_+$ the triangular decomposition of $\mathfrak{g}$.
Let $(\cdot,\cdot)$ be the nondegenerate even supersymmetric invariant bilinear form
which is normalised by $h^\vee\geq0$ and $(\alpha,\alpha)=2$ for a root of maximal square length,
where $h^{\vee}$ is the dual Coxeter number, i.e., the half of the eigenvalue of the Casimir operator.
Set $\alpha^{\vee}=\alpha$ for simple isotropic root $\alpha\in\Pi$, and $\alpha^{\vee}=2\alpha/(\alpha,\alpha)$
for simple non-isotropic root $\alpha\in\Pi$.
Let $\rho$ be the Weyl vector of $\mathfrak{g}$ respect to $\Pi$, i.e, $(\rho,\alpha_i)=\frac{1}{2}(\alpha_i,\alpha_i)$ for any $i=1,\cdots,n$,
we have $h^\vee=(\rho,\theta)+\frac{1}{2}(\theta,\theta)$.

Let $\widetilde{\mathfrak{g}}=\mathfrak{g}\otimes\mathbb{C}[t,t^{-1}]\oplus\mathbb{C}k$ be the affine Lie superalgebra of $\mathfrak{g}$,
we identify $\mathfrak{g}$ with $\mathfrak{g}\otimes t^0$ and
set $a(n)=a\otimes t^n$ for $a\in\mathfrak{g}$, $n\in\mathbb{Z}$ for convenience.
Define subalgebras
$N_+=\mathfrak{n}_+\oplus\mathfrak{g}\otimes t\mathbb{C}[t]$,
$B=N_+\oplus\mathfrak{h}\oplus\mathbb{C}k$ and $P=\mathfrak{g}\otimes\mathbb{C}[t]\oplus\mathbb{C}k$.
Let $\widehat{\mathfrak{g}}=\widetilde{\mathfrak{g}}\oplus\mathbb{C}d$ be the extended affine Lie superalgebra
and $\widehat{\mathfrak{h}}=\mathfrak{h}\oplus\mathbb{C}k\oplus\mathbb{C}d$ the Cartan subalgebra of $\widehat{\mathfrak{g}}$.
Let $\delta$ be the linear function on $\widehat{\mathfrak{h}}$ defined by $\delta|_{\mathfrak{h}\oplus\mathbb{C}k}=0, \delta(d)=1$.
Then $\widehat{\Pi}=\{\alpha_0=\delta-\theta,\alpha_1,\cdots,\alpha_n\}$ is a set of simple roots of $\widehat{\mathfrak{g}}$.
Define the fundamental weights $\Lambda_i\in \widehat{\mathfrak{h}}^* (i=0,1,\cdots,n)$ by $\Lambda_i(\alpha_j^{\vee})=\delta_{ij}, \Lambda_i(d)=0$ for any $i,j\in\{0,1,\cdots,n\}$.
Note that $\widehat{\rho}=h^{\vee}\Lambda_0+\rho$ is the Weyl vector of $\widehat{\mathfrak{g}}$ respect to $\widehat{\Pi}$.
Let $r_\alpha$ be the reflection with respect to an even root $\alpha$,
set $r_{\alpha}.\hat{\lambda}=r_{\alpha}(\hat{\lambda}+\widehat{\rho})-\widehat{\rho}$ for $\hat{\lambda}\in \widehat{\mathfrak{h}}^*$.
For $\beta\in \mathfrak{h}^*$, define the translation $t_\beta\in \mbox{End}~\widehat{\mathfrak{h}}^*$ by
$$t_\beta(\hat{\lambda})=\hat{\lambda}+\hat{\lambda}(k)\beta-((\hat{\lambda},\beta)+\frac{1}{2}\hat{\lambda}(k)(\beta,\beta))\delta.$$

For any $\hat{\lambda}\in \widehat{\mathfrak{h}}^*$,
denote by $M(\hat{\lambda})$(resp. $L(\hat{\lambda})$) the Verma (resp. irreducible highest weight) $\widehat{\mathfrak{g}}$-module.
We say that a $\widehat{\mathfrak{g}}$-module $M$ is from the category $\mathcal{O}$ if
the Cartan subalgebra $\widehat{\mathfrak{h}}$ acts semisimply on $M$ with finite dimensional weight spaces
and finite number of maximal weights.
It is clear that any $\widehat{\mathfrak{g}}$-module is also a $\widetilde{\mathfrak{g}}$-module,
denote by $L(\mathcal{k},\lambda)$ the $\widetilde{\mathfrak{g}}$-module $L(\hat{\lambda})$, $M(\mathcal{k},\lambda)$ the $\widetilde{\mathfrak{g}}$-module $M(\hat{\lambda})$, where $\mathcal{k}=\hat{\lambda}(k),\lambda=\hat{\lambda}|_{\mathfrak{h}}$.
For convenience, we set $\Lambda_i=\Lambda_i|_{\mathfrak{h}}(i=1,\cdots,n)$.
For any irreducible highest weight $\widehat{\mathfrak{g}}$-module $L(\hat{\lambda})$,
$L(\mathcal{k},\lambda)$ is also irreducible as $\widetilde{\mathfrak{g}}$-module.
Conversely, for any restricted $\widetilde{\mathfrak{g}}$-module $M$ of level $\mathcal{k}\neq-1$,
by the Sugawara construction we can extend $M$ to a $\widehat{\mathfrak{g}}$-module by letting $d$ acts on $M$ as $-L(0)$.
In this paper we shall consider any restricted $\widetilde{\mathfrak{g}}$-module as a $\widehat{\mathfrak{g}}$-module in this way.

Let $\mathcal{k}\in\mathbb{C}$ and $U$ a $\mathfrak{h}$-module,
$U$ can be regarded as a $B$-module by letting $N_+$ acts trivially and $k$ acts as $\mathcal{k}$,
let $M(\mathcal{k},U)=U(\widetilde{\mathfrak{g}})\otimes_{U(B)}U$.
If $U=\mathbb{C}$ is a one-dimensional $\mathfrak{h}$-module on which $h\in\mathfrak{h}$ acts as a fixed complex number $\lambda(h)$ with $\lambda\in\mathfrak{h}^*$,
the corresponding module is a Verma module denoted by $M(\mathcal{k},\lambda)$.
Then $M(\mathcal{k},\lambda)$ has an unique maximal submodule and $L(\mathcal{k},\lambda)$ is isomorphic to the corresponding irreducible quotient.
Similarly we can define the generalized Verma $\tilde{\mathfrak{g}}$-module $V(\mathcal{k},U)=U(\tilde{\mathfrak{g}})\otimes_{U(P)}U$
for any $\mathfrak{g}$-module $U$ which can be extend to a $P$-module by letting $\mathfrak{g}\otimes t\mathbb{C}[t]$
acts trivially and $k$ acts as $\mathcal{k}$.
Note that if $U=\mathbb{C}$ is the trivial $\mathfrak{g}$-module,
then $V(\mathcal{k},\mathbb{C})$ is a quotient of Verma module $M(\mathcal{k},0)$
and $L(\mathcal{k},0)$ is isomorphic to the irreducible quotient of $V(\mathcal{k},\mathbb{C})$.

For finite-dimensional simple Lie algebra $\mathfrak{g}$,
let $X=\{\lambda\in \widehat{\mathfrak{h}}^*\mid (\lambda,\alpha)\geq 0$ for all but finite number of $\alpha\in \widehat{\Delta}_+^{re}$
and $(\lambda, \alpha)\ne 0$ for any isotropic root $\alpha\}$,
$R_\lambda=\{\alpha\in\widehat{\Delta}^{re}\mid \langle \lambda, \alpha^\vee\rangle\in\mathbb{Z}\}$
and $R_\lambda^+=R_\lambda\cap\widehat{\Delta}_+^{re}$.
We call $\lambda\in \widehat{\mathfrak{h}}^*$ an admissible weight if
$\lambda+\widehat{\rho}\in X$, $\langle \lambda+\widehat{\rho}, \alpha^\vee\rangle>0$ for all $\alpha\in R_\lambda^+$
and $\mathbb{Q}R_\lambda=\mathbb{Q}\widehat{\Delta}^{re}$ (cf. \cite{KW88}).
A number $\mathcal{k}\in\mathbb{C}$ is called an admissible level for $\widehat{\mathfrak{g}}$ if $\mathcal{k}\Lambda_0$ is an admissible weight,
and $\mathcal{k}$ is called a principal admissible level if $R_{\mathcal{k}\Lambda_0}$ is isometric to $\widehat{\Delta}^{re}$.
An admissible weight $\lambda$ is called principal admissible if its level $\mathcal{k}$ is principal admissible and $R_\lambda$
is isometric to $R_{\mathcal{k}\Lambda_0}$.
The admissible weight was extended to Lie superalgebra in \cite{KRW03,KW14,GK25}.
It is clear that the definitions of boundary principal admissible level in \cite{KRW03,KW14,GK25} are coincide.
Let $\mathfrak{g}$ be a basic classical Lie superalgebra except $D(2|1;a)$,
$\mathcal{k}$ a boundary principal admissible level of $\widehat{\mathfrak{g}}$ and
$\lambda$ a principal admissible weight of level $\mathcal{k}$ in sense of \cite{GK25},
from Theorem 7.3.2 and Proposition 10.2 in \cite{GK25},
we obtain that $\Delta_{\lambda+\widehat{\rho}}$ is isometric to $\Delta_{\mathcal{k}\Lambda_0+\widehat{\rho}}$,
 then $\lambda$ is a principal admissible weight of level $\mathcal{k}$ in sense of \cite{KRW03,KW14},
the converse is obvious, hence all principal admissible weights of level $\mathcal{k}$ in sense of \cite{GK25}
coincide with that in \cite{KRW03,KW14}.
All boundary admissible levels of indecomposable non-twisted affine Kac-Moody superalgebras except $D(2|1;a)^{(1)}$ are classified in \cite{GK25}.

\begin{lem}[\cite{GK25}]\label{boundary}
Let $\mathfrak{g}$ be a basic classical Lie superalgebra
except $D(2|1;a)$, all boundary admissible level of $\widehat{\mathfrak{g}}$ are\\
(1) Principal: $\mathcal{k}=\frac{h^{\vee}}{u}-h^{\vee}$, where $u$ is a positive integer such that $(u,r^\vee)=(u,h^\vee)=1$,
where $r^\vee$ is lacety of $\mathfrak{g}^\sharp$,
i.e., the ratio of square lengths of a long root and of a short root of $\triangle^\sharp=\{\alpha\in\Delta^0~|~(\alpha,\alpha)>0\}$ (if $h^\vee$ is not integral, then $(u,h^\vee)=1$ stands for $(2u,2h^\vee)=1$);\\
(2) Subprincipal: this happens only for $\mathfrak{g}=B(0|n)$, in this case $\mathcal{k}=\frac{2n-1}{2u}-h^{\vee}$,
where $u$ is an even positive integer which is coprime with $2n-1$.\\
Moreover, if $h^\vee=0$, then $\widehat{\mathfrak{g}}$ does not admit boundary admissible levels.
\end{lem}
For positive integer $u$, let $\widehat{\Pi}_u=\{u\delta-\theta,\alpha_1,\cdots,\alpha_n\}$.
Then we have the following lemma, which gives all boundary principal admissible weights.
\begin{lem}[\cite{KW89,KW14}]\label{lemboundary}
Let $\mathfrak{g}$ be a basic classical Lie superalgebra, $\mathcal{k}=\frac{h^{\vee}}{u}-h^{\vee}$ a boundary principal admissible level of $\widehat{\mathfrak{g}}$,
then all principal admissible weights of level $\mathcal{k}$ are of the form (up to adding a multiple of $\delta$)
$$\Lambda=(t_\beta y).(\mathcal{k}\Lambda_0),$$
where $\beta\in \mathfrak{h}^*, y\in W$ such that $(t_\beta y)\widehat{\Pi}_u\subset\widehat{\Delta}_+$.
\end{lem}
Since $(t_\beta y)(u\delta-\theta)=-y\theta+(u+(\beta,y\theta))\delta$ and
$(t_\beta y)\alpha_i=y\alpha_i-(\beta,y\alpha_i)\delta$, where $i=1,\cdots,n$,
then we have $(t_\beta y)\widehat{\Pi}_u\subset\widehat{\Delta}_+$ if and only if
$\begin{cases}
   u+(\beta,y\theta)\in\mathbb{Z}_{\geq0}, & \text{if}~ y\theta\in -\Delta_+,\\
   u+(\beta,y\theta)\in\mathbb{Z}_{>0},  & \text{if } y\theta\in \Delta_+,
\end{cases}$
and for any $i=1,\cdots,n$,
$\begin{cases}
  -(\beta,y\alpha_i)\in\mathbb{Z}_{\geq0}, & \text{if}~ y\alpha_i\in \Delta_+,\\
  -(\beta,y\alpha_i)\in\mathbb{Z}_{>0},  & \text{if } y\alpha_i\in -\Delta_+.
\end{cases}$

\section{Affine Vertex operator superalgebra}
\label{sec:3}
	\def\theequation{3.\arabic{equation}}
	\setcounter{equation}{0}
In this section, we first recall some concepts and facts about $\mathbb{Z}$-graded vertex operator superalgebra.
Then for affine vertex operator superalgebra $L_{\widehat{\mathfrak{g}}}(\mathcal{k},0)$,
we give the criterion for determining that $L(\hat{\lambda})$ is an irreducible ordinary $L_{\widehat{\mathfrak{g}}}(\mathcal{k},0)$-module.

\subsection{$\mathbb{Z}$-graded vertex operator superalgebra}
\begin{defi}
{\em  Let $V$ be a $\mathbb{Z}$-graded vertex operator superalgebra.
A {\em weak $V$-module} is a module for $V$ as vertex superalgebra.
A weak $V$-module $M$ is called a {\em $\mathbb{Z}_{\geq 0}$-graded} {\em weak module}
if $M$ has a $\mathbb{Z}_ {\geq 0}$-gradation $M=\bigoplus_{n\in\mathbb{Z}_{\geq 0}}M(n)$ such that
$a_mM(n)\subseteq M(\mbox{wt}a+n-m-1)$
for any homogeneous element $a\in V, m\in\mathbb{Z}, n\in\mathbb{Z}_{\geq 0}$.
A weak $V$-module $M$ is called an {\em ordinary module}
if $M=\bigoplus_{\lambda\in\mathbb{C}}M_{\lambda}$, where $M_{\lambda}=\{w\in M \mid L(0)w=\lambda w\}$ such that
$\mbox{dim}~M_{\lambda}<\infty$ for all $\lambda\in\mathbb{C}$ and
$M_{\lambda}=0$ for the real part of $\lambda$ sufficiently small.}
\end{defi}

Let $V$ be a $\mathbb{Z}$-graded vertex operator superalgebra.
For any homogeneous element $a\in V$, we define
\begin{equation*}
a\ast b=\mbox{Res}_z\frac{(1+z)^{{\rm wt}a}}{z}Y(a,z)b
\end{equation*}
for any $b\in V$,
then we can extend $\ast$ on $V$.
Let $O(V)$ be the subspace of $V$ linearly spanned by
\begin{equation*}
\mbox{Res}_z\frac{(1+z)^{{\rm wt}a}}{z^2}Y(a,z)b
\end{equation*}
for any homogeneous element $a\in V$ and any $b\in V$.
Let $M$ be any weak $V$-module, we define
$\Omega(M)=\{u\in M\mid a_m u=0~~\mbox{for}~~a\in V, m>\mbox{wt}a-1\}$,
and $o$ to be the linear map from $V$ to $\mbox{End}~\Omega(M)$ such that $o(a)=a_{{\rm wt}a-1}$
for any homogeneous element $a\in V$.
Then we have the following theorem.
\begin{thm}[\cite{FZ92,Zhu96}]\label{thmzhu}
Let $V$ be a $\mathbb{Z}$-graded vertex operator superalgebra and $M$ a weak $V$-module.\\
{\rm (a)} The subspace $O(V)$ is a two-sided ideal of $V$ with respect to the product $\ast$
and $A(V)=V/O(V)$ is an associative superalgebra with identity $\textbf{1}+O(V)$.
Moreover, $\omega+O(V)$ lies in the center of $A(V)$.\\
{\rm (b)} $\Omega(M)$ is an $A(V)$-module with $a$ acts as $o(a)$ for any $a\in V$.\\
{\rm (c)} There is an induction functor $L$ from the category of $A(V)$-modules to the category of $\mathbb{Z}_{\geq 0}$-graded weak $V$-modules.
Moreover, $\Omega(L(U))=U$ for any $A(V)$-module $U$.\\
{\rm (d)} $\Omega$ and $L$ are inverse bijections between the sets of irreducible modules in each category.\\
\end{thm}

\subsection{Ordinary $L_{\widehat{\mathfrak{g}}}(\mathcal{k},0)$-modules }
We know that for $\mathcal{k}\neq -h^\vee$, both
$V(\mathcal{k},\mathbb{C})$ and $L(\mathcal{k},0)$ have natural
$\mathbb{Z}$-graded vertex operator superalgebra structures
and any $\widetilde{\mathfrak{g}}$-module $M(\mathcal{k},U)$ is a weak module for $V(\mathcal{k},\mathbb{C})$,
we denote by $V_{\widehat{\mathfrak{g}}}(\mathcal{k},0)$ the universal affine vertex operator superalgebra $V(\mathcal{k},\mathbb{C})$
and $L_{\widehat{\mathfrak{g}}}(\mathcal{k},0)$ the simple affine vertex operator superalgebra $L(\mathcal{k},0)$.
Let $\mathbf{1}$ be the vacuum vector of $V_{\widehat{\mathfrak{g}}}(\mathcal{k},0)$.
For $a\in V_{\widehat{\mathfrak{g}}}(\mathcal{k},0)$, denote by $\overline{a}$ the image of $a$ under the projection of $V_{\widehat{\mathfrak{g}}}(\mathcal{k},0)$ onto $A(V_{\widehat{\mathfrak{g}}}(\mathcal{k},0))$.
Similar to Theorem 3.1.1 of \cite{FZ92}, we have the following lemma.
\begin{lem}\label{lemzhu}
The associative superalgebra $A(V_{\widehat{\mathfrak{g}}}(\mathcal{k},0))$ is canonically isomorphic to $U(\mathfrak{g})$.
The isomorphism is given by
\begin{equation*}
\begin{aligned}
F:~~~~~~~~~~~~~~~~~~~~~~A(V_{\widehat{\mathfrak{g}}}(\mathcal{k},0))&\rightarrow U(\mathfrak{g})\\
\overline{a_1(-n_1-1)\cdots a_m(-n_m-1)\mathbf{1}}&\mapsto (-1)^{\sum_{1\leq i<j\leq m}|a_i||a_j|+\sum_{i=1}^{m} n_i}a_m\cdots a_1,
\end{aligned}
\end{equation*}
where $a_1,\cdots,a_m\in\mathfrak{g},n_1,\cdots,n_m\in\mathbb{Z}_{\geq 0}$.
\end{lem}

For the weak $L_{\widehat{\mathfrak{g}}}(\mathcal{k},0)$-modules in the category $\mathcal{O}$, we have the following theorem.
Arakawa proved the case for $\mathfrak{g}$ being any finite dimensional simple Lie algebra in \cite{Ar16};
Gorelik and Kac proved the corresponding results for $\mathfrak{g}$ being any basic classical Lie superalgebra
except $D(2|1;a)$ in \cite{GK25}.
\begin{thm}[\cite{Ar16,GK25}]\label{thmar16}
Let $\mathfrak{g}$ be a basic classical Lie superalgebra
except $D(2|1;a)$, $\mathcal{k}$ a boundary principal admissible level of $\widehat{\mathfrak{g}}$ and $\hat{\lambda}\in \widehat{\mathfrak{h}}^*$.
Then $L(\hat{\lambda})$ is a weak $L_{\widehat{\mathfrak{g}}}(\mathcal{k},0)$-module if and only if
$\hat{\lambda}$ is a principal admissible weight of level $\mathcal{k}$.
\end{thm}

Then we can furthermore present the characterization of irreducible ordinary modules based on the Theorem \ref{thmar16}.
\begin{lem}\label{lemor}
Let $\mathfrak{g}$ be a basic classical Lie superalgebra
except $D(2|1;a)$, $\mathcal{k}$ a boundary principal admissible level of $\widehat{\mathfrak{g}}$.
Then the set $\{L(\hat{\lambda})~|~\hat{\lambda}\in \widehat{\mathfrak{h}}^*$ is a principal admissible weight of level $\mathcal{k}$ such that $L(\lambda)$
is a finite-dimensional $\mathfrak{g}$-module, where $\lambda=\hat{\lambda}|_{\mathfrak{h}}\}$
provides the complete list of irreducible ordinary $L_{\widehat{\mathfrak{g}}}(\mathcal{k},0)$-modules.
\end{lem}
\begin{proof}
Let $M$ be an irreducible ordinary $L_{\widehat{\mathfrak{g}}}(\mathcal{k},0)$-module, then
$M=\bigoplus_{h\in\mathbb{Z}_{\geq 0}}M_{h+\alpha}$ satisfying $\alpha\in\mathbb{C}$ and
$\mbox{dim}~M_{h+\alpha}<\infty$ for any $h\in\mathbb{Z}_{\geq 0}$.
For any $w\in M_{h+\alpha}$,
since $M$ is a $\tilde{\mathfrak{g}}$-module
and $\mbox{wt}~aw=\mbox{wt}~w$ for any $a \in\mathfrak{g}$,
we have $M_{h+\alpha}$ is a finite dimensional $\mathfrak{g}$-module.
From the Proposition 5.4 of \cite{DLM98},
we have that $\Omega(M)=M_{\alpha}$ is an irreducible $A(L_{\widehat{\mathfrak{g}}}(\mathcal{k},0))$-module,
then $M_{\alpha}$ is a finite-dimensional irreducible $\mathfrak{g}$-module.
By the Proposition 1.39 of \cite{CW12}, $M_{\alpha}$ is an irreducible highest weight $\mathfrak{g}$-module,
i.e., $M_{\alpha}=L(\lambda)$ for some $\lambda\in \mathfrak{h}^*$, thus $M=L(\hat{\lambda})$ satisfying $\mathcal{k}=\hat{\lambda}(k),\lambda=\hat{\lambda}|_{\mathfrak{h}}$.
From Theorem \ref{thmar16}, $\hat{\lambda}$ is a principal admissible weight.
Conversely, let $\hat{\lambda}\in \widehat{\mathfrak{h}}^*$ be a principal admissible weight of level $\mathcal{k}$ such that $L(\lambda)$
is a finite-dimensional $\mathfrak{g}$-module,
it shows that $L(\hat{\lambda})$ is an irreducible ordinary $L_{\widehat{\mathfrak{g}}}(\mathcal{k},0)$-module.
\end{proof}

\section{Classification of irreducible ordinary modules for affine vertex operator superalgebra}
\label{sec:4}
	\def\theequation{4.\arabic{equation}}
	\setcounter{equation}{0}
In this section, we consider the affine vertex operator superalgebra associated to
finite dimensional simple Lie algebra, basic classical Lie superalgebras of type I and type II.
If $\mathfrak{g}$ is a basic classical Lie superalgebra of type I and
$\mathcal{k}=\frac{h^{\vee}}{u}-h^{\vee}$ a boundary admissible level of $\widehat{\mathfrak{g}}$,
we prove that $L_{\widehat{\mathfrak{g}}}(\mathcal{k},0)$ has exactly $u$ inequivalent irreducible ordinary modules.
If $\mathfrak{g}$ is a finite dimensional simple Lie algebra or a basic classical Lie superalgebra of type II and
$\mathcal{k}$ a boundary admissible level of $\widehat{\mathfrak{g}}$ ($\mathcal{k}$ is principal for $\mathfrak{g}=osp(1|2n)$),
we prove that $L_{\widehat{\mathfrak{g}}}(\mathcal{k},0)$ itself is the only irreducible ordinary $L_{\widehat{\mathfrak{g}}}(\mathcal{k},0)$-module.

\subsection{Finite-dimensional simple Lie algebra}
Let $\mathfrak{g}$ be a finite dimensional simple Lie algebra, note that $h^\vee=(\rho,\theta)+1$.

\begin{lem}\label{lemy1}
Let $\mathfrak{g}$ be a finite dimensional simple Lie algebra. For any $1\ne y\in W$, there exists $\alpha\in Q_+$ satisfying $\theta+y^{-1}\alpha\in Q_+$,
and
$\begin{cases}
  (\rho,\alpha-y^{-1}\alpha)\geq h^\vee, & \text{if}~y^{-1}\theta\in -\Delta_+,\\
  (\rho,\alpha-y^{-1}\alpha)> h^\vee,  & \text{if }~y^{-1}\theta\in \Delta_+.
\end{cases}$
\end{lem}
\begin{proof}
Note that $y^{-1}\theta\in \Delta$.
If $y^{-1}\theta\in -\Delta_+$, let $\alpha=\theta\in Q_+$,
then we have $\theta+y^{-1}\theta\in Q_+$.
Since $\theta$ is a long root, it shows that $y^{-1}\theta$ is also a long root, thus $(\rho,-y^{-1}\theta)\geq1$, then we have
$(\rho,\theta-y^{-1}\theta)\geq(\rho,\theta)+1=h^\vee$.
If $y^{-1}\theta\in \Delta_+\setminus\{\theta\}$,
we have $(\rho,\theta-y^{-1}\theta)>0$.
We take integer $k>\frac{h^\vee}{(\rho,\theta-y^{-1}\theta)}$ and let $\alpha=k\theta\in Q_+$,
then we get $\theta+y^{-1}\alpha=\theta+ky^{-1}\theta\in Q_+$ and
$(\rho,\alpha-y^{-1}\alpha)=k(\rho,\theta-y^{-1}\theta)> h^\vee$.
If $y^{-1}\theta=\theta$, since $y\ne1$, then $(y\rho, \alpha_j)<0$ for certain $j$,
thus $(\rho-y\rho, \alpha_j)>0$, i.e., $(\rho,\alpha_j-y^{-1}\alpha_j)>0$.
We take integer $m>\frac{h^\vee}{(\rho,\alpha_j-y^{-1}\alpha_j)}$ and $\alpha=k\theta+m\alpha_j$,
then we get $(\rho,\alpha-y^{-1}\alpha)=(\rho,k\theta+m\alpha_j-ky^{-1}\theta-my^{-1}\alpha_j)=m(\rho,\alpha_j-y^{-1}\alpha_j)> h^\vee$.
Furthermore, let $k$ sufficiently large, then we have $\alpha=k\theta+m\alpha_j\in Q_+$ and
$\theta+y^{-1}\alpha=(k+1)\theta+my^{-1}\alpha_j\in Q_+$.
\end{proof}

From Lemma \ref{boundary}, all boundary admissible levels of $\widehat{\mathfrak{g}}$ are principal, then we have the following theorem.

\begin{thm}
Let $\mathfrak{g}$ be a finite dimensional simple Lie algebra and
$\mathcal{k}=\frac{h^{\vee}}{u}-h^{\vee}$ a boundary admissible level of $\widehat{\mathfrak{g}}$.
Then $L_{\widehat{\mathfrak{g}}}(\mathcal{k},0)$ itself is the only irreducible ordinary $L_{\widehat{\mathfrak{g}}}(\mathcal{k},0)$-module.
\end{thm}
\begin{proof}
Let $L(\hat{\lambda})$ be an irreducible ordinary $L_{\widehat{\mathfrak{g}}}(\mathcal{k},0)$-module,
from Lemma \ref{lemor}, $\hat{\lambda}\in \widehat{\mathfrak{h}}^*$ is a principal admissible weight of level $\mathcal{k}$ such that $L(\lambda)$
is a finite-dimensional $\mathfrak{g}$-module, where $\lambda=\hat{\lambda}|_{\mathfrak{h}}$.
From Lemma \ref{lemboundary}, we obtain $\hat{\lambda}=(t_\beta y).(\mathcal{k}\Lambda_0)$,
where $\beta\in \mathfrak{h}^*, y\in W$
satisfying if $y\theta\in -\Delta_+$, $u+(\beta,y\theta)\in\mathbb{Z}_{\geq0}$,
if $y\theta\in \Delta_+$, $u+(\beta,y\theta)\in\mathbb{Z}_{>0}$,
if $y\alpha_i\in \Delta_+$, $-(\beta,y\alpha_i)\in\mathbb{Z}_{\geq0}$, and
if $y\alpha_i\in -\Delta_+$, $-(\beta,y\alpha_i)\in\mathbb{Z}_{>0}$,
where $i=1,\cdots,n$.
Then from Theorem 21.1 of \cite{Hum78}, we obtain
$(\hat{\lambda}, \alpha_i^\vee)\in \mathbb{Z}_{\geq0}$, i.e., $\frac{h^\vee}{u}(\beta,\alpha_i)+(\rho,y^{-1}\alpha_i-\alpha_i)\in \frac{(\alpha_i,\alpha_i)}{2}\mathbb{Z}_{\geq0}$, where $i=1,\cdots,n$.

If $y^{-1}\theta\in -\Delta_+$, by Lemma \ref{lemy1},
we have $u-(\beta,\theta)= u+(\beta,y\theta)-(\beta,y(\theta+y^{-1}\theta))\in\mathbb{Z}_{\geq0}$
and $\frac{h^\vee}{u}(\beta,\theta)+(\rho,y^{-1}\theta-\theta)\geq0$,
then $\frac{h^\vee}{u}(\beta,\theta)+(\rho,y^{-1}\theta-\theta)\leq \frac{h^\vee}{u}(\beta,\theta)-h^\vee
=\frac{h^\vee}{u}((\beta,\theta)-u)\leq 0$,
thus $\frac{h^\vee}{u}(\beta,\theta)+(\rho,y^{-1}\theta-\theta)=0$,
i.e., $(\hat{\lambda}, \theta)=0$. This shows that if $\beta$ exists,
we have $(\hat{\lambda}, \alpha_i)=0$ for any $i=1,\cdots,n$, then $\hat{\lambda}|_{\mathfrak{h}}=0$, up to adding a multiple of $\delta$, $\hat{\lambda}=\mathcal{k}\Lambda_0$.
If $y\ne 1$ and $y^{-1}\theta\in \Delta_+$, by Lemma \ref{lemy1}, there exists $\alpha\in Q_+$ satisfying $\theta+y^{-1}\alpha\in Q_+$ and $(\rho,\alpha-y^{-1}\alpha)> h^\vee$,
then we have $u-(\beta,\alpha)= u+(\beta,y\theta)-(\beta,y(\theta+y^{-1}\alpha))\in\mathbb{Z}_{\geq0}$ and $\frac{h^\vee}{u}(\beta,\alpha)+(\rho,y^{-1}\alpha-\alpha)\geq0$, thus $\frac{h^\vee}{u}(\beta,\alpha)+(\rho,y^{-1}\alpha-\alpha)< \frac{h^\vee}{u}(\beta,\alpha)-h^\vee
=\frac{h^\vee}{u}((\beta,\alpha)-u)\leq 0$,
this is impossible.
If $y=1$, then we have $(\beta,\alpha_i)=0$ for any $i=1,\cdots,n$, thus $\beta=0$, i.e., $\hat{\lambda}=\mathcal{k}\Lambda_0$.
This implies that $L_{\widehat{\mathfrak{g}}}(\mathcal{k},0)$ itself is the only irreducible ordinary $L_{\widehat{\mathfrak{g}}}(\mathcal{k},0)$-module.
\end{proof}

\subsection{Basic classical Lie superalgebra of type I}
Basic classical Lie superalgebras of type I consist of type $A(n,m)$ and type $C(n+1)$.
If $\mathfrak{g}$ is of type $A(n,n)$, then $h^\vee=0$, from Lemma \ref{boundary},
$\widehat{\mathfrak{g}}$ does not admit boundary admissible levels,
thus we consider the type $A(n,m)$ with $n\ne m$.

\subsubsection{Type $A(n-1,m-1)$}
Let $\mathfrak{g}$ be the Lie superalgebra of type $A(n-1,m-1)$,
that is, $\mathfrak{g}=sl(n|m)$, where $m,n>0$ and $n\ne m$.
Note that when $n\ne m$, $sl(n|m)\cong sl(m|n)$,
we assume $n>m>0$ in this paper.
We fix the Dynkin diagram
$$
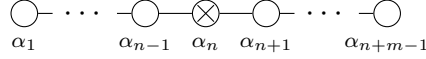
\begin{tikzpicture}[
    baseline=-0.5ex,
    line width=0.4pt,
    root/.style={circle, draw, minimum size=10pt, inner sep=0pt},
    otimes/.style={circle, draw, minimum size=10pt, inner sep=0pt,
        path picture={
            \draw (-3.1pt,-3.1pt) -- (3.1pt, 3.1pt);
            \draw (-3.1pt, 3.1pt) -- (3.1pt,-3.1pt);
        }
    }
]
    \node[root, label=below:{$\scriptstyle \alpha_1$}]      (n1) at (0, 0) {};
    \node (dots1) at (0.8, 0) {$\cdots$};
    \node[root, label=below:{$\scriptstyle \alpha_{n-1}$}]  (n2) at (1.6, 0) {};
    \node[otimes, label=below:{$\scriptstyle \alpha_n$}]    (n3) at (2.4, 0) {};
    \node[root, label=below:{$\scriptstyle \alpha_{n+1}$}]  (n4) at (3.2, 0) {};
    \node (dots2) at (4.0, 0) {$\cdots$};
    \node[root, label=below:{$\scriptstyle \alpha_{n+m-1}$}] (n5) at (4.8, 0) {};
    \draw (n1) -- (dots1) -- (n2) -- (n3) -- (n4) -- (dots2) -- (n5);
\end{tikzpicture}
$$
and $\Pi=\{\alpha_1,\cdots,\alpha_{n+m-1}\}$,
where $\alpha_1,\cdots,\alpha_{n-1},\alpha_{n+1},\cdots,\alpha_{n+m-1}$ are even and $\alpha_n$ is odd.
Then the corresponding Cartan matrix is
$$A=\begin{pmatrix}
2 & -1  & \cdots & 0 & 0 & 0 & \cdots & 0 & 0\\
-1 & 2  & \cdots & 0 & 0 & 0 & \cdots & 0 & 0\\
\vdots & \vdots & \ddots & \vdots & \vdots & \vdots & & \vdots & \vdots\\
0 & 0 & \cdots & 2 & -1 & 0 & \cdots & 0 & 0\\
0 & 0  & \cdots & -1 & 0 & 1 & \cdots & 0 & 0\\
0 & 0 & \cdots & 0 & -1 & 2 & \cdots & 0 & 0\\
\vdots & \vdots  & & \vdots & \vdots &\vdots & \ddots&\vdots & \vdots \\
0 & 0  & \cdots & 0 & 0 & 0 & \cdots&2 & -1\\
0 & 0  & \cdots & 0 & 0 & 0 & \cdots&-1 & 2
\end{pmatrix}
.$$
Note that $h^{\vee}=n-m$, $r^\vee=1$, $\theta=\alpha_1+\cdots+\alpha_{n+m-1}$,
$(\rho, \alpha_i)=1$ for $i=1,\cdots,n-1$, $(\rho, \alpha_n)=0$ and $(\rho, \alpha_i)=-1$ for $i=n+1,\cdots,n+m-1$.
We have $W=W_1\times W_2$, where $W_1=\langle r_{\alpha_i}|i=1,\cdots,n-1\rangle$ and $W_2=\langle r_{\alpha_i}|i=n+1,\cdots,n+m-1\rangle$.
Let $\theta_1=\alpha_1+\cdots+\alpha_{n-1}$.
We can identify $W_1$ with the Weyl group of $sl_n$ and identify $\theta_1$ with the highest root of $sl_n$.
Similarly, $\Delta^0=\Delta_1\cup\Delta_2$, where $\Delta_1$ is the root system of $sl_n$ and $\Delta_2$ is the root system of $sl_m$.
Then we have the following lemma.
\begin{lem}\label{lemy3}
Let $\mathfrak{g}=sl(n|m)$, where $n>m>0$.
If $y\notin W_2$,
there exists $\alpha\in Q_{1,+}$ satisfying $\theta+y^{-1}\alpha\in Q_+$ and $(\rho,\alpha-y^{-1}\alpha)\geq n-m$,
where $Q_{1,+}=\bigoplus_{i=1}^{n-1}\mathbb{Z}_{\geq0}\alpha_i$.
\end{lem}
\begin{proof}
Let $y=y_1y_2$, where $1\ne y_1\in W_1$ and $y_2\in W_2$, then $y^{-1}\theta_1=y_1^{-1}\theta_1\in \Delta_1$.
If $y_1^{-1}\theta_1\in -\Delta_{1,+}$,
where $\Delta_{1,+}$ is the positive root system of $sl_n$,
let $\alpha=\theta_1\in Q_{1,+}$,
then we have $\theta+y^{-1}\theta_1=\alpha_n+\cdots+\alpha_{n+m-1}+\theta_1+y_1^{-1}\theta_1\in Q_+$ and
$(\rho,\theta_1-y^{-1}\theta_1)>(\rho,\theta_1)=n-1\geq n-m$.
If $y_1^{-1}\theta_1\in \Delta_{1,+}\setminus\{\theta_1\}$,
then $(\rho,\theta_1-y_1^{-1}\theta_1)>0$,
we take integer $k\geq\frac{n-m}{(\rho,\theta_1-y_1^{-1}\theta_1)}$ and let $\alpha=k\theta_1\in Q_{1,+}$,
then we get $\theta+y^{-1}\alpha=\theta+ky_1^{-1}\theta_1\in Q_+$ and
$(\rho,\alpha-y^{-1}\alpha)=k(\rho,\theta_1-y_1^{-1}\theta_1)\geq n-m$.
If $y_1^{-1}\theta_1=\theta_1$, since $y_1\ne1$, then $(y_1\rho, \alpha_j)<0$ for certain $j\in\{1,\cdots,n-1\}$,
thus $(\rho-y_1\rho, \alpha_j)>0$, i.e., $(\rho,\alpha_j-y_1^{-1}\alpha_j)>0$.
We take integer $l\geq\frac{n-m}{(\rho,\alpha_j-y_1^{-1}\alpha_j)}$ and $\alpha=k\theta_1+l\alpha_j$,
then we get $(\rho,\alpha-y^{-1}\alpha)=(\rho,k\theta_1+l\alpha_j-ky_1^{-1}\theta_1-ly_1^{-1}\alpha_j)=l(\rho,\alpha_j-y_1^{-1}\alpha_j)\geq n-m$.
Furthermore, let $k$ sufficiently large, then we have $\alpha=k\theta_1+l\alpha_j\in Q_{1,+}$ and
$\theta+y^{-1}\alpha=\alpha_n+\cdots+\alpha_{n+m-1}+(k+1)\theta_1+ly_1^{-1}\alpha_j\in Q_+$.
\end{proof}

For the extended affine Lie superalgebra $\widehat{\mathfrak{g}}$,
the corresponding Cartan matrix is
$$\widehat{A}=\begin{pmatrix}
0&-1 & 0  & \cdots & 0 & 0 & 0 & \cdots & 0 & 1\\
-1&2 & -1  & \cdots & 0 & 0 & 0 & \cdots & 0 & 0\\
0&-1 & 2  & \cdots & 0 & 0 & 0 & \cdots & 0 & 0\\
\vdots &\vdots & \vdots & \ddots & \vdots & \vdots & \vdots & & \vdots & \vdots\\
0&0 & 0 & \cdots & 2 & -1 & 0 & \cdots & 0 & 0\\
0&0 & 0  & \cdots & -1 & 0 & 1 & \cdots & 0 & 0\\
0&0 & 0 & \cdots & 0 & -1 & 2 & \cdots & 0 & 0\\
\vdots & \vdots & \vdots & &\vdots & \vdots & \vdots& \ddots& \vdots & \vdots \\
0&0 & 0  & \cdots & 0 & 0 & 0 & \cdots&2 & -1\\
-1&0 & 0  & \cdots & 0 & 0 & 0 & \cdots&-1 & 2
\end{pmatrix}
.$$
From Lemma \ref{boundary}, all boundary admissible levels of $\widehat{\mathfrak{g}}$ are principal,
thus the main result for type $A(n-1,m-1)$ is the following theorem.

\begin{thm}
Let $\mathfrak{g}=sl(n|m)~(n>m>0)$ and
$\mathcal{k}=\frac{n-m}{u}-(n-m)$ be a boundary admissible level of $\widehat{\mathfrak{g}}$.
Then the set $\{L(\hat{\lambda})~|~\hat{\lambda}=(\frac{p+1}{u}-1)(n-m)\Lambda_0-\frac{p}{u}(n-m)\Lambda_n, p=0,\cdots,u-1\}$
provides the complete list of irreducible ordinary $L_{\widehat{\mathfrak{g}}}(\mathcal{k},0)$-modules.
\end{thm}
\begin{proof}
Let $L(\hat{\lambda})$ be an irreducible ordinary $L_{\widehat{\mathfrak{g}}}(\mathcal{k},0)$-module,
by Lemma \ref{lemor}, we know that $\hat{\lambda}\in \widehat{\mathfrak{h}}^*$ is a principal admissible weight of level $\mathcal{k}$ such that $L(\lambda)$
is a finite-dimensional $\mathfrak{g}$-module, where $\lambda=\hat{\lambda}|_{\mathfrak{h}}$.
Since $y\theta\in \Delta_+$ for any $y\in W$, from Lemma \ref{lemboundary}, we obtain $\hat{\lambda}=(t_\beta y).(\mathcal{k}\Lambda_0)$,
where $\beta\in \mathfrak{h}^*, y\in W$ satisfying
$u+(\beta,y\theta)\in\mathbb{Z}_{>0}$,
and for any $i=1,\cdots,n+m-1$,
if $y\alpha_i\in \Delta_+$, $-(\beta,y\alpha_i)\in\mathbb{Z}_{\geq0}$,
if $y\alpha_i\in -\Delta_+$, $-(\beta,y\alpha_i)\in\mathbb{Z}_{>0}$.
Then from Theorem 8 of \cite{Kac77},
we have $(\hat{\lambda}, \alpha_i)\in \mathbb{Z}_{\geq0} ~(i=1,\cdots,n-1)$, $-(\hat{\lambda}, \alpha_i)\in \mathbb{Z}_{\geq0} ~(i=n+1,\cdots,n+m-1)$,
i.e.,
$$\begin{cases}
 \frac{n-m}{u}(\beta,\alpha_i)+(\rho,y^{-1}\alpha_i-\alpha_i)\in \mathbb{Z}_{\geq0}, & i=1,\cdots,n-1,\\
 -\frac{n-m}{u}(\beta,\alpha_i)-(\rho,y^{-1}\alpha_i-\alpha_i)\in \mathbb{Z}_{\geq0},  &i=n+1,\cdots,n+m-1.
\end{cases}$$

If $1\ne y\in W_2$, from Corollary 10.2 of \cite{Hum78}, there exist $\alpha_i ~(i\in\{n+1,\cdots,n+m-1\})$ satisfying $y^{-1}\alpha_i\in -\Delta_{2,+}$,
thus we have $(\beta,\alpha_i)=-(\beta,y(-y^{-1}\alpha_i))\in\mathbb{Z}_{\geq0}$ and $-\frac{n-m}{u}(\beta,\alpha_i)-(\rho,y^{-1}\alpha_i-\alpha_i)\in \mathbb{Z}_{\geq0}$.
Since $y^{-1}\alpha_i-\alpha_i\in-Q_{2,+}$, we have $(\rho,y^{-1}\alpha_i-\alpha_i)>0$,
this is impossible.
If $y\notin W_2$, by Lemma \ref{lemy3}, there exists $\alpha\in Q_{1,+}$ satisfying $\theta+y^{-1}\alpha\in Q_+$ and $(\rho,\alpha-y^{-1}\alpha)\geq n-m$,
then we have $u-(\beta,\alpha)=u+(\beta,y\theta)-(\beta,y(\theta+y^{-1}\alpha))\in\mathbb{Z}_{>0}$ and $\frac{n-m}{u}(\beta,\alpha)+(\rho,y^{-1}\alpha-\alpha)\in \mathbb{Z}_{\geq0}$, thus $(\beta,\alpha)\geq\frac{u}{n-m}(\rho,\alpha-y^{-1}\alpha)\geq u$,
this is also impossible.
If $y=1$, then we have
\begin{equation}\label{222}
\begin{cases}
  u+(\beta,\theta)\in\mathbb{Z}_{>0},\\
  (\beta,\alpha_i)=0, &i=1,\cdots,n-1,\\
  -(\beta,\alpha_n)\in\mathbb{Z}_{\geq0},\\
  -(\beta,\alpha_i)\in u\mathbb{Z}_{\geq0},&i=n+1,\cdots,n+m-1.\\
\end{cases}
\end{equation}
Since $-(\beta,\theta)<u$ and $-(\beta,\alpha_i)=0~(i=1,\cdots,n-1),
  -(\beta,\alpha_n)\in\mathbb{Z}_{\geq0},
  -(\beta,\alpha_i)\in u\mathbb{Z}_{\geq0}~(i=n+1,\cdots,n+m-1)$,
we have $(\beta,\alpha_i)=0$ for any $i=n+1,\cdots,n+m-1$.
Let $\beta=k_1\alpha_1+\cdots+k_{n+m-1}\alpha_{n+m-1}$,
then (\ref{222}) is equivalent to
$$\begin{cases}
  u-k_{n-1}+k_{n+1}\in\mathbb{Z}_{>0},&\\
  2k_1-k_2=0,&\\
  -k_{i-1}+2k_i-k_{i+1}=0,~~i=2,\cdots,n-1,n+1,\cdots,n+m-2,\\
   k_{n+m-2}-2k_{n+m-1}=0,\\
  k_{n-1}-k_{n+1}\in\mathbb{Z}_{\geq0}.\\
\end{cases}$$
Thus we get $\beta=k_1\alpha_1+2k_1\alpha_2+\cdots+nk_1\alpha_n+\frac{(m-1)n}{m}k_1\alpha_{n+1}+\cdots+\frac{n}{m}k_1\alpha_{n+m-1},$ where $k_1=0,\frac{m}{n-m},\cdots,\frac{m(u-1)}{n-m}$.
Then $\hat{\lambda}=t_\beta.(\mathcal{k}\Lambda_0)$, where $\beta=k_1\alpha_1+2k_1\alpha_2+\cdots+nk_1\alpha_n+\frac{(m-1)n}{m}k_1\alpha_{n+1}+\cdots+\frac{n}{m}k_1\alpha_{n+m-1}, k_1=0,\frac{m}{n-m},\cdots,\frac{m(u-1)}{n-m}$,
up to adding a multiple of $\delta$,
we have $\hat{\lambda}=(\frac{p+1}{u}-1)(n-m)\Lambda_0-\frac{p}{u}(n-m)\Lambda_n,~\mbox{where}~p=0,\cdots,u-1.$
\end{proof}

\subsubsection{Type $C(n+1)$}
Let $\mathfrak{g}$ be the Lie superalgebra of type $C(n+1)$,
that is, $\mathfrak{g}=osp(2|2n)$, where $n>0$.
We fix the Dynkin diagram
$$
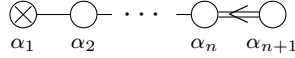
\begin{tikzpicture}[
    baseline=-0.5ex,
    line width=0.4pt,
    root/.style={circle, draw, minimum size=10pt, inner sep=0pt},
    otimes/.style={circle, draw, minimum size=10pt, inner sep=0pt,
        path picture={
            \draw (-3.1pt,-3.1pt) -- (3.1pt, 3.1pt);
            \draw (-3.1pt, 3.1pt) -- (3.1pt,-3.1pt);}}]
    \node[otimes, label=below:{$\scriptstyle \alpha_1$}]    (n3) at (0, 0) {};
    \node[root, label=below:{$\scriptstyle \alpha_{2}$}]  (n4) at (0.8, 0) {};
    \node (dots2) at (1.6, 0) {$\cdots$};
    \node[root, label=below:{$\scriptstyle \alpha_{n}$}] (n5) at (2.4, 0) {};
     \node[root, label=below:{$\scriptstyle \alpha_{n+1}$}] (n6) at (3.3, 0) {};
    \draw (n3) -- (n4) -- (dots2) -- (n5);
    \draw[double, double distance=1.5pt] (n5) -- (n6) node[midway] {$<$};
\end{tikzpicture}
$$
and $\Pi=\{\alpha_1,\alpha_2,\cdots,\alpha_{n+1}\}$,
where $\alpha_2,\cdots,\alpha_{n+1}$ are even and $\alpha_1$ is odd.
Then the corresponding Cartan matrix is
$$A=\begin{pmatrix}
 0 & -\frac{1}{2} &0& \cdots &0& 0 & 0\\
 -1 & 2 &-1& \cdots &0& 0 & 0\\
 0 & -1 &2& \cdots &0& 0 & 0\\
 \vdots &  \vdots &\vdots & \ddots& \vdots& \vdots & \vdots \\
  0 & 0 &0& \cdots&2&-1 & 0\\
  0 & 0 &0& \cdots&-1&2 & -2\\
 0 & 0 &0& \cdots&0&-1 & 2
\end{pmatrix}
,$$
and $\theta=\alpha_1+2\alpha_2+\cdots+2\alpha_{n}+\alpha_{n+1}$,
$W=\langle r_{\alpha_i}|i=2,\cdots,n+1\rangle$.
Note that $h^{\vee}=n$, $r^\vee=2$ if $n>1$ and $r^\vee=1$ if $n=1$, and $(\rho, \alpha_1)=0$, $(\rho, \alpha_i)=\frac{1}{2}$ for $i=2,\cdots,n$, $(\rho, \alpha_{n+1})=1$.
Let $\theta_0=\theta-\alpha_1=2\alpha_2+\cdots+2\alpha_{n}+\alpha_{n+1}$.
We can identify $W$ with the Weyl group of $sp_{2n}$, identify $\Delta^0$ with the root system of $sp_{2n}$ and identify $\theta_0$ with the highest root of $sp_{2n}$.
For the extended affine Lie superalgebra $\widehat{\mathfrak{g}}$,
the corresponding Cartan matrix is
$$\widehat{A}=\begin{pmatrix}
0&1&-\frac{1}{2}&0&\cdots &0& 0 & 0\\
 1&0 & -\frac{1}{2} &0& \cdots &0& 0 & 0\\
-1& -1 & 2 &-1& \cdots &0& 0 & 0\\
0& 0 & -1 &2& \cdots &0& 0 & 0\\
  \vdots &\vdots &  \vdots &\vdots & \ddots& \vdots& \vdots & \vdots \\
0&  0 & 0 &0& \cdots&2&-1 & 0\\
 0& 0 & 0 &0& \cdots&-1&2 & -2\\
0& 0 & 0 &0& \cdots&0&-1 & 2
\end{pmatrix}
.$$
From Lemma \ref{boundary}, all boundary admissible levels of $\widehat{\mathfrak{g}}$ are principal,
then we have the following theorem for type $C(n+1)$.

\begin{thm}
Let $\mathfrak{g}=osp(2|2n)$ and
$\mathcal{k}=\frac{n}{u}-n$ be a boundary admissible level of $\widehat{\mathfrak{g}}$.
Then the set $\{L(\hat{\lambda})~|~\hat{\lambda}=(\frac{p+1}{u}-1)n\Lambda_0-\frac{np}{u}\Lambda_1,~\mbox{where}~p=0,1,\cdots,u-1\}$
provides the complete list of irreducible ordinary $L_{\widehat{\mathfrak{g}}}(\mathcal{k},0)$-modules.
\end{thm}
\begin{proof}
Let $L(\hat{\lambda})$ be an irreducible ordinary $L_{\widehat{\mathfrak{g}}}(\mathcal{k},0)$-module,
by Lemma \ref{lemor}, we know that $\hat{\lambda}\in \widehat{\mathfrak{h}}^*$ is a principal admissible weight of level $\mathcal{k}$ such that $L(\lambda)$
is a finite-dimensional $\mathfrak{g}$-module, where $\lambda=\hat{\lambda}|_{\mathfrak{h}}$.
Since $y\theta\in \Delta_+$ for all $y\in W$, from Lemma \ref{lemboundary}, we obtain $\hat{\lambda}=(t_\beta y).(\mathcal{k}\Lambda_0)$,
where $\beta\in \mathfrak{h}^*, y\in W$ satisfying
$u+(\beta,y\theta)\in\mathbb{Z}_{>0}$,
and for any $i=1,\cdots,n+1$,
if $y\alpha_i\in \Delta_+$, $-(\beta,y\alpha_i)\in\mathbb{Z}_{\geq0}$,
if $y\alpha_i\in -\Delta_+$, $-(\beta,y\alpha_i)\in\mathbb{Z}_{>0}$.
Then from Theorem 8 of \cite{Kac77}, we obtain $(\hat{\lambda}, \alpha_i^\vee)\in \mathbb{Z}_{\geq0}$,
i.e., $\frac{n}{u}(\beta,\alpha_i)+(\rho,y^{-1}\alpha_i-\alpha_i)\in \frac{(\alpha_i,\alpha_i)}{2}\mathbb{Z}_{\geq0}$,
where $i=2,\cdots,n+1$.

If $y\ne 1$, similar to Lemma \ref{lemy3}, there exists $\alpha\in Q_{+}^0$ satisfying $\theta+y^{-1}\alpha\in Q_+$ and $(\rho,\alpha-y^{-1}\alpha)> n$,
then we have $u-(\beta,\alpha)=u+(\beta,y\theta)-(\beta,y(\theta+y^{-1}\alpha))\in\mathbb{Z}_{>0}$ and $\frac{n}{u}(\beta,\alpha)+(\rho,y^{-1}\alpha-\alpha)\geq0$,
thus $(\beta,\alpha)\geq\frac{u}{n}(\rho,\alpha-y^{-1}\alpha)\geq u$,
this is impossible.
If $y=1$, then we have
\begin{equation}\label{333}
\begin{cases}
  u+(\beta,\theta)\in\mathbb{Z}_{>0},&\\
  (\beta,\alpha_i)=0, &i=2,\cdots,n+1,\\
  -(\beta,\alpha_1)\in\mathbb{Z}_{\geq0}.\\
\end{cases}
\end{equation}
Let $\beta=k_1\alpha_1+\cdots+k_{n+1}\alpha_{n+1}$,
then (\ref{333}) is equivalent to
$$\begin{cases}
  u-k_1+\frac{1}{2}k_2\in\mathbb{Z}_{>0},\\
  \frac{1}{2}k_2\in\mathbb{Z}_{\geq0},\\
  -k_{i-1}+2k_i-k_{i+1}=0, &i=2,\cdots,n-1,\\
   -k_{n-1}+2k_n-2k_{n+1}=0,\\
  -k_n+2k_{n+1}=0.\\
\end{cases}$$
Thus we get $\beta=k_1\alpha_1+k_1\alpha_2+\cdots+k_1\alpha_n+\frac{1}{2}k_1\alpha_{n+1},$ where $k_1=0,2,\cdots,2(u-1)$.
Then $\hat{\lambda}=t_\beta.(\mathcal{k}\Lambda_0)$, where $\beta=k_1\alpha_1+k_1\alpha_2+\cdots+k_1\alpha_n+\frac{1}{2}k_1\alpha_{n+1}, k_1=0,2,\cdots,2(u-1)$,
up to adding a multiple of $\delta$,
we have $\hat{\lambda}=(\frac{p+1}{u}-1)n\Lambda_0-\frac{np}{u}\Lambda_1,~\mbox{where}~p=0,1,\cdots,u-1.$
\end{proof}
\subsection{Basic classical Lie superalgebra of type II}
Basic classical Lie superalgebras of type II consist of type $B(m,n)$, type $D(m,n)$, type $F(4)$, type $G(3)$ and type $D(2|1;a)$.
\subsubsection{Type $B(0,n)$}
Let $\mathfrak{g}$ be the Lie superalgebra of type $B(0,n)$,
that is, $\mathfrak{g}=osp(1|2n)$, where $n>0$.
We fix the Dynkin diagram
$$
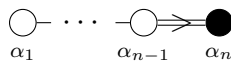
\begin{tikzpicture}[
    baseline=-0.5ex,
    line width=0.4pt,
    root/.style={circle, draw, minimum size=10pt, inner sep=0pt},
    otimes/.style={circle, draw, minimum size=10pt, inner sep=0pt,
        path picture={
            \draw (-3.1pt,-3.1pt) -- (3.1pt, 3.1pt);
            \draw (-3.1pt, 3.1pt) -- (3.1pt,-3.1pt);}},
   blackroot/.style={circle, draw, fill=black, minimum size=10pt, inner sep=0pt}]
    \node[root, label=below:{$\scriptstyle \alpha_1$}]    (n3) at (0, 0) {};
    \node (dots2) at (0.8, 0) {$\cdots$};
    \node[root, label=below:{$\scriptstyle \alpha_{n-1}$}] (n5) at (1.6, 0) {};
     \node[blackroot, label=below:{$\scriptstyle \alpha_{n}$}] (n6) at (2.6, 0) {};
    \draw (n3) -- (dots2) -- (n5);
    \draw[double, double distance=1.5pt] (n5) -- (n6) node[midway] {$>$};
\end{tikzpicture}
$$
and $\Pi=\{\alpha_1,\cdots,\alpha_{n}\}$,
where $\alpha_1,\cdots,\alpha_{n-1}$ are even and $\alpha_n$ is odd.
Then $\theta=2\alpha_1+2\alpha_2+\cdots+2\alpha_{n}$ and
$W=\langle r_{\alpha_i},r_{2\alpha_n}\mid i=2,\cdots,n-1\rangle$.
Note that $h^{\vee}=n+\frac{1}{2}$, $r^\vee=2$ if $n>1$ and $r^\vee=1$ if $n=1$, and $(\rho, \alpha_i)=\frac{1}{2}$ for $i=1,\cdots,n-1$, $(\rho, \alpha_n)=\frac{1}{4}$.
We can identify $W$ with the Weyl group of $sp_{2n}$, identify $\Delta^0$ with the root system of $sp_{2n}$ and identify $\theta$ with the highest root of $sp_{2n}$.

\begin{lem}\label{lemy5}
Let $\mathfrak{g}=osp(1|2n)$.
For any $1\ne y\in W$, there exists $\alpha\in Q_+$ satisfying $\theta+y^{-1}\alpha\in Q_+$,
and $\begin{cases}
  (\rho,\alpha-y^{-1}\alpha)\geq n, & \text{if}~y^{-1}\theta\in -\Delta_+^0,\\
  (\rho,\alpha-y^{-1}\alpha)> n+\frac{1}{2},  & \text{if }~y^{-1}\theta\in \Delta_+^0.
\end{cases}$
\end{lem}
\begin{proof}
Since $W$ is the Weyl group of $sp_{2n}$, $\Delta^0$ is the root system of $sp_{2n}$ and $\theta$ is the highest root of $sp_{2n}$,
we have $y^{-1}\theta\in \Delta^0$.
If $y^{-1}\theta\in -\Delta_{+}^0$, let $\alpha=\theta\in Q_{+}$,
then we get $\theta+y^{-1}\theta\in Q_+$
and $(\rho,\theta-y^{-1}\theta)\geq(\rho,\theta)+\frac{1}{2}=n$.
If $y^{-1}\theta\in \Delta_+^0\setminus\{\theta\}$,
then we have $(\rho,\theta-y^{-1}\theta)>0$,
we take integer $k>\frac{n+\frac{1}{2}}{(\rho,\theta-y^{-1}\theta)}$ and let $\alpha=k\theta\in Q_+$,
then we have $\theta+y^{-1}\alpha=\theta+ky^{-1}\theta\in Q_+$ and
$(\rho,\alpha-y^{-1}\alpha)=k(\rho,\theta-y^{-1}\theta)> n+\frac{1}{2}$.
If $y^{-1}\theta=\theta$, since $y\ne1$, then we get $(y\rho, \alpha_j)<0$ for certain $j=1,\cdots,n$,
thus $(\rho-y\rho, \alpha_j)>0$, i.e., $(\rho,\alpha_j-y^{-1}\alpha_j)>0$.
We take integer $m>\frac{n+\frac{1}{2}}{(\rho,\alpha_j-y^{-1}\alpha_j)}$ and let $\alpha=k\theta+m\alpha_j$,
then we have $(\rho,\alpha-y^{-1}\alpha)=(\rho,k\theta+m\alpha_j-ky^{-1}\theta-my^{-1}\alpha_j)=m(\rho,\alpha_j-y^{-1}\alpha_j)> n+\frac{1}{2}$.
Furthermore, let $k$ sufficiently large, then we have $\alpha=k\theta+m\alpha_j\in Q_+$ and
$\theta+y^{-1}\alpha=(k+1)\theta+my^{-1}\alpha_j\in Q_+$.
\end{proof}
Thus the main result for type $B(0,n)$ is the following theorem.
\begin{thm}
Let $\mathfrak{g}=osp(1|2n)$ and
$\mathcal{k}=\frac{2n+1}{2u}-(n+\frac{1}{2})$ be a boundary principal admissible level of $\widehat{\mathfrak{g}}$.
Then $L_{\widehat{\mathfrak{g}}}(\mathcal{k},0)$ itself is the only irreducible ordinary $L_{\widehat{\mathfrak{g}}}(\mathcal{k},0)$-module.
\end{thm}
\begin{proof}
Let $L(\hat{\lambda})$ be an irreducible ordinary $L_{\widehat{\mathfrak{g}}}(\mathcal{k},0)$-module,
from Lemma \ref{lemboundary} and Lemma \ref{lemor}, we obtain $\hat{\lambda}=(t_\beta y).(\mathcal{k}\Lambda_0)$
satisfying if $y\theta\in -\Delta_+$, $u+(\beta,y\theta)\in\mathbb{Z}_{\geq0}$,
if $y\theta\in \Delta_+$, $u+(\beta,y\theta)\in\mathbb{Z}_{>0}$,
and for any $i=1,\cdots,n$,
if $y\alpha_i\in \Delta_+$, $-(\beta,y\alpha_i)\in\mathbb{Z}_{\geq0}$,
if $y\alpha_i\in -\Delta_+$, $-(\beta,y\alpha_i)\in\mathbb{Z}_{>0}$,
where $\beta\in \mathfrak{h}^*, y\in W$.
Then from Theorem 8 of \cite{Kac77},
we obtain
$(\hat{\lambda}, \alpha_i^\vee)\in \mathbb{Z}_{\geq0}$ and $(\hat{\lambda}, \alpha_n^\vee)\in 2\mathbb{Z}_{\geq0}$,
 i.e., $\frac{n+\frac{1}{2}}{u}(\beta,\alpha_i)+(\rho,y^{-1}\alpha_i-\alpha_i)\in \frac{1}{2}\mathbb{Z}_{\geq0}$,
 where $i=1,\cdots,n$.

If $y^{-1}\theta\in -\Delta_+^0$, by Lemma \ref{lemy5},
we have $u-(\beta,\theta)= u+(\beta,y\theta)-(\beta,y(\theta+y^{-1}\theta))\in\mathbb{Z}_{\geq0}$
and $\frac{n+\frac{1}{2}}{u}(\beta,\theta)+(\rho,y^{-1}\theta-\theta)\in\mathbb{Z}_{\geq0}$,
then $\frac{n+\frac{1}{2}}{u}(\beta,\theta)+(\rho,y^{-1}\theta-\theta)\leq \frac{n+\frac{1}{2}}{u}(\beta,\theta)-(n+\frac{1}{2})+\frac{1}{2}
=\frac{n+\frac{1}{2}}{u}((\beta,\theta)-u)+\frac{1}{2}\leq \frac{1}{2}$,
thus $\frac{n+\frac{1}{2}}{u}(\beta,\theta)+(\rho,y^{-1}\theta-\theta)=0$,
i.e., $(\hat{\lambda}, \theta)=0$. This shows that if $\beta$ exists,
we have $(\hat{\lambda}, \alpha_i)=0$ for any $i=1,\cdots,n$,
then $\hat{\lambda}|_{\mathfrak{h}}=0$, up to adding a multiple of $\delta$, $\hat{\lambda}=\mathcal{k}\Lambda_0$.
If $y\ne 1$ and $y^{-1}\theta\in \Delta_+^0$, by Lemma \ref{lemy5}, there exists $\alpha\in Q_+$ satisfying $\theta+y^{-1}\alpha\in Q_+$ and $(\rho,\alpha-y^{-1}\alpha)>n+\frac{1}{2}$,
then we have $u-(\beta,\alpha)= u+(\beta,y\theta)-(\beta,y(\theta+y^{-1}\alpha))\in\mathbb{Z}_{\geq0}$ and $\frac{n+\frac{1}{2}}{u}(\beta,\alpha)+(\rho,y^{-1}\alpha-\alpha)\geq0$, thus $\frac{n+\frac{1}{2}}{u}(\beta,\alpha)+(\rho,y^{-1}\alpha-\alpha)< \frac{h^\vee}{u}(\beta,\alpha)-(n+\frac{1}{2})
=\frac{n+\frac{1}{2}}{u}((\beta,\alpha)-u)\leq 0$,
this is impossible.
If $y=1$, then we have $(\beta,\alpha_i)=0$ for any $i=1,\cdots,n$, thus $\beta=0$, i.e., $\hat{\lambda}=\mathcal{k}\Lambda_0$.
This implies that $L_{\widehat{\mathfrak{g}}}(\mathcal{k},0)$ itself is the only irreducible ordinary $L_{\widehat{\mathfrak{g}}}(\mathcal{k},0)$-module.
\end{proof}

\subsubsection{Type $B(m,n)$}

Let $\mathfrak{g}$ be the Lie superalgebra of type $B(m,n)$,
that is, $\mathfrak{g}=osp(2m+1|2n)$, where $m,n>0$.
We fix the Dynkin diagram
$$
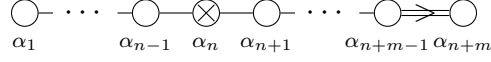
\begin{tikzpicture}[
    baseline=-0.5ex,
    line width=0.4pt,
    root/.style={circle, draw, minimum size=10pt, inner sep=0pt},
    otimes/.style={circle, draw, minimum size=10pt, inner sep=0pt,
        path picture={
            \draw (-3.1pt,-3.1pt) -- (3.1pt, 3.1pt);
            \draw (-3.1pt, 3.1pt) -- (3.1pt,-3.1pt);}},
   blackroot/.style={circle, draw, fill=black, minimum size=10pt, inner sep=0pt}]
    \node[root, label=below:{$\scriptstyle \alpha_1$}]    (n1) at (0, 0) {};
    \node (dots1) at (0.8, 0) {$\cdots$};
     \node[root, label=below:{$\scriptstyle \alpha_{n-1}$}]    (n2) at (1.6, 0) {};
    \node[otimes, label=below:{$\scriptstyle \alpha_{n}$}] (n3) at (2.4, 0) {};
      \node[root, label=below:{$\scriptstyle \alpha_{n+1}$}]    (n4) at (3.2, 0) {};
       \node (dots2) at (4, 0) {$\cdots$};
     \node[root, label=below:{$\scriptstyle \alpha_{n+m-1}$}] (n5) at (4.8, 0) {};
      \node[root, label=below:{$\scriptstyle \alpha_{n+m}$}] (n6) at (5.8, 0) {};
    \draw (n1) -- (dots1) -- (n2)--(n3)--(n4)--(dots2)--(n5);
    \draw[double, double distance=1.5pt] (n5) -- (n6) node[midway] {$>$};
\end{tikzpicture}
$$
and $\Pi=\{\alpha_1,\cdots,\alpha_{n+m}\}$,
where $\alpha_1,\cdots,\alpha_{n-1},\alpha_{n+1},\cdots,\alpha_{n+m}$ are even and $\alpha_n$ is odd.
Note that $\theta=2\alpha_1+\cdots+2\alpha_{n+m}$.
Let $\alpha_n^\prime=2\alpha_n+\cdots+2\alpha_{n+m}$.
We have $W=W_1\times W_2$, where $W_1=\langle r_{\alpha_i}, r_{\alpha_n^\prime}\mid i=1,\cdots,n-1\rangle$ and $W_2=\langle r_{\alpha_i}|i=n+1,\cdots,n+m\rangle$,
we can identify $W_1$ with the Weyl group of $sp_{2n}$, identify $\theta$ with the highest root of $sp_{2n}$,
identify $W_2$ with the Weyl group of $o_{2m+1}$ and identify $\theta^\prime=\alpha_{n+1}+2\alpha_{n+2}+\cdots+2\alpha_{n+m}$ with the highest root of $o_{2m+1}$.
We also have $\Delta^0=\Delta_1\cup\Delta_2$, and we can identify $\Delta_1$ with the root system of $sp_{2n}$ and
identify $\Delta_2$ with the root system of $o_{2m+1}$.
From Lemma \ref{boundary}, all boundary admissible levels of $\widehat{\mathfrak{g}}$ are principal.
In the following, we classify the irreducible ordinary modules for $L_{\widehat{\mathfrak{g}}}(\mathcal{k},0)$
by considering two cases, one is $m>n$, the other is $m\leq n$.

{\bf Case 1: $m>n$.} 
Note that $h^{\vee}=2(m-n)-1$, $r^\vee=2$,
$(\rho, \alpha_i)=-1$ for $i=1,\cdots,n-1$, $(\rho, \alpha_n)=0$,
$(\rho, \alpha_i)=1$ for $i=n+1,\cdots,n+m-1$ and $(\rho, \alpha_{n+m})=\frac{1}{2}$.
Thus the main result for type $B(m,n)~(m>n)$ is the following theorem.
\begin{thm}\label{thmbmn}
Let $\mathfrak{g}=osp(2m+1|2n)~(m>n>0)$ and
$\mathcal{k}=\frac{2(m-n)-1}{u}-2(m-n)+1$ be a boundary admissible level of $\widehat{\mathfrak{g}}$.
Then $L_{\widehat{\mathfrak{g}}}(\mathcal{k},0)$ itself is the only irreducible ordinary $L_{\widehat{\mathfrak{g}}}(\mathcal{k},0)$-module.
\end{thm}
\begin{proof}
Let $L(\hat{\lambda})$ be an irreducible ordinary $L_{\widehat{\mathfrak{g}}}(\mathcal{k},0)$-module,
from Lemma \ref{lemboundary} and Lemma \ref{lemor}, we obtain $\hat{\lambda}=(t_\beta y).(\mathcal{k}\Lambda_0)$
satisfying if $y\theta\in -\Delta_+$, $u+(\beta,y\theta)\in\mathbb{Z}_{\geq0}$,
if $y\theta\in \Delta_+$, $u+(\beta,y\theta)\in\mathbb{Z}_{>0}$,
and for any $i=1,\cdots,n+m$,
if $y\alpha_i\in \Delta_+$, $-(\beta,y\alpha_i)\in\mathbb{Z}_{\geq0}$,
if $y\alpha_i\in -\Delta_+$, $-(\beta,y\alpha_i)\in\mathbb{Z}_{>0}$,
where $\beta\in \mathfrak{h}^*, y\in W$.
Then from Theorem 8 of \cite{Kac77}, we obtain
$(\hat{\lambda}, \alpha_i^\vee)\in \mathbb{Z}_{\geq0} ~(i=1,\cdots,n-1,n+1,\cdots,n+m)$ and $(\hat{\lambda}, {\alpha_n^\prime}^\vee)\in \mathbb{Z}_{\geq0}$,
i.e.,
$$\begin{cases}
  \frac{2(m-n)-1}{u}(\beta,\alpha_i)+(\rho,y^{-1}\alpha_i-\alpha_i)\in -\mathbb{Z}_{\geq0},&i=1,\cdots,n-1,\\
 \frac{2(m-n)-1}{u}(\beta,\alpha_n^\prime)+(\rho,y^{-1}\alpha_n^\prime-\alpha_n^\prime)\in -2\mathbb{Z}_{\geq0},\\
  \frac{2(m-n)-1}{u}(\beta,\alpha_i)+(\rho,y^{-1}\alpha_i-\alpha_i)\in \mathbb{Z}_{\geq0},&i=n+1,\cdots,n+m-1,\\
 \frac{2(m-n)-1}{u}(\beta,\alpha_{n+m})+(\rho,y^{-1}\alpha_{n+m}-\alpha_{n+m})\in \frac{1}{2}\mathbb{Z}_{\geq0}.\\
\end{cases}$$

If $y\notin W_1$, similar to Lemma \ref{lemy3}, there exists $\alpha\in Q_{2,+}$ satisfying $\theta+y^{-1}\alpha\in Q_+$ and $(\rho,\alpha-y^{-1}\alpha)>2(m-n)-1$, where $Q_{2,+}=\bigoplus_{i=n+1}^{n+m}\mathbb{Z}_{\geq0}\alpha_i$.
Then we have $u-(\beta,\alpha)=u+(\beta,y\theta)-(\beta,y(\theta+y^{-1}\alpha))\in\mathbb{Z}_{\geq0}$ and $\frac{2(m-n)-1}{u}(\beta,\alpha)+(\rho,y^{-1}\alpha-\alpha)\geq0$, thus $(\beta,\alpha)\geq\frac{u}{2(m-n)-1}(\rho,\alpha-y^{-1}\alpha)>u$,
this is impossible.
If $1\ne y\in W_1^\prime$, where $W_1^\prime=\langle r_{\alpha_i}|i=1,\cdots,n-1\rangle$.
Note that $W_1^\prime$ is the Weyl group of $sl_{n}$, let $\Delta_1^\prime\subset\Delta_1$ be the corresponding root system,
then from Corollary 10.2 of \cite{Hum78}, there exist $\alpha_i ~(i\in\{1,\cdots,n-1\})$ satisfying $y^{-1}\alpha_i\in -\Delta_{1,+}^\prime$,
thus we have $(\beta,\alpha_i)=-(\beta,y(-y^{-1}\alpha_i))\in\mathbb{Z}_{\geq0}$ and $-\frac{2(m-n)-1}{u}(\beta,\alpha_i)-(\rho,y^{-1}\alpha_i-\alpha_i)\in \mathbb{Z}_{\geq0}$,
while we note that $(\rho,y^{-1}\alpha_i-\alpha_i)>0$, this is also impossible.
If $y\in W_1\backslash W_1^\prime$,
denote by \(\Delta_{1,+}^l=\{ 2\sum_{i=k}^{n-1} \alpha_i + \alpha_n^\prime ~(k=1, \dots, n-1), \alpha_n^\prime\}\)
the set of positive long roots of $\Delta_1$, then all long roots in \(\Delta_{1,+}^l\)
form an orthonormal basis of the root space of $sp_{2n}$,
and \(W_1'\) acts transitively on \(\Delta_{1,+}^l\), thus
$W_1' $ is the stabilizer subgroup of $\Delta_{1,+}^l$ in $W_1$.
For any \(y \in W_1\setminus W_1'\) we have \(y^{-1}\Delta_{1,+}^l \neq \Delta_{1,+}^l\),
so there exists \(\alpha \in \Delta_{1,+}^l\) with \(y^{-1}\alpha \notin \Delta_{1,+}^l\).
Since \(y^{-1}\) preserves root lengths, we have \(y^{-1}\alpha\) is a long root and therefore negative, i.e. \(y^{-1}\alpha \in -\Delta_{1,+}^l\).
For any long root $\gamma\in \Delta_1$, it is clear that $(\rho, \gamma)$ is odd, thus $(\rho,\alpha-y^{-1}\alpha)$ is even.
Suppose that $y\alpha_i\in \Delta_+$ for any $\alpha_i$ appeared in $-y^{-1}\alpha$,
then $-\alpha=y(-y^{-1}\alpha)\in \Delta_+$, it is impossible, thus $(\beta, \alpha)=-(\beta,y(-y^{-1}\alpha))\in \mathbb{Z}_{>0}$.
Then we have $u-(\beta,\alpha)=u+(\beta,y\theta)-(\beta,y(\theta+y^{-1}\alpha))\in\mathbb{Z}_{\geq0}$,
and $(\beta,\alpha)\in \frac{2u}{2(m-n)-1}(\frac{(\rho,\alpha-y^{-1}\alpha)}{2}-\mathbb{Z}_{\geq0})$, this is also impossible.
If $y=1$, then we have
$$\begin{cases}
u+(\beta,\theta)\in\mathbb{Z}_{>0},\\
(\beta,\alpha_i)\in-\mathbb{Z}_{\geq0},&i=1,\cdots,n+m,\\
(\beta,\alpha_i)\in -\frac{u}{2(m-n)-1}\mathbb{Z}_{\geq0},&i=1,\cdots,n-1,\\
(\beta,\alpha_n^\prime)\in -\frac{2u}{2(m-n)-1}\mathbb{Z}_{\geq0},\\
(\beta,\alpha_i)\in \frac{u}{2(m-n)-1}\mathbb{Z}_{\geq0},&i=n+1,\cdots,n+m-1,\\
(\beta,\alpha_{n+m})\in \frac{u}{4(m-n)-2}\mathbb{Z}_{\geq0},
\end{cases}$$
note that $(u,2(m-n)-1)=(u,2)=1$, we have $(\beta,\alpha_i)=0$ for $i=1,\cdots,n+m$, thus $\beta=0$, i.e., $\hat{\lambda}=\mathcal{k}\Lambda_0$.
This implies that $L_{\widehat{\mathfrak{g}}}(\mathcal{k},0)$ itself is the only irreducible ordinary $L_{\widehat{\mathfrak{g}}}(\mathcal{k},0)$-module.
\end{proof}

{\bf Case 2: $m\leq n$.} 
Note that $h^{\vee}=n-m+\frac{1}{2}$, $r^\vee=2$ if $n>1$ and $r^\vee=1$ if $n=1$,
$(\rho, \alpha_i)=\frac{1}{2}$ for $i=1,\cdots,n-1$, $(\rho, \alpha_n)=0$,
$(\rho, \alpha_i)=-\frac{1}{2}$ for $i=n+1,\cdots,n+m-1$ and $(\rho, \alpha_{n+m})=-\frac{1}{4}$.
Thus the main result for type $B(m,n)~(m\leq n)$ is the following theorem.

\begin{thm}\label{thmbmn1}
Let $\mathfrak{g}=osp(2m+1|2n)~(0<m\leq n)$ and
$\mathcal{k}=\frac{2(n-m)+1}{2u}-(n-m+\frac{1}{2})$ be a boundary admissible level of $\widehat{\mathfrak{g}}$.
Then $L_{\widehat{\mathfrak{g}}}(\mathcal{k},0)$ itself is the only irreducible ordinary $L_{\widehat{\mathfrak{g}}}(\mathcal{k},0)$-module.
\end{thm}
\begin{proof}
Let $L(\hat{\lambda})$ be an irreducible ordinary $L_{\widehat{\mathfrak{g}}}(\mathcal{k},0)$-module,
from Lemma \ref{lemor} and  Lemma \ref{lemboundary}, we have $\hat{\lambda}=(t_\beta y).(\mathcal{k}\Lambda_0)$
satisfying if $y\theta\in -\Delta_+$, $u+(\beta,y\theta)\in\mathbb{Z}_{\geq0}$,
if $y\theta\in \Delta_+$, $u+(\beta,y\theta)\in\mathbb{Z}_{>0}$,
and for any $i=1,\cdots,n+m$,
if $y\alpha_i\in \Delta_+$, $-(\beta,y\alpha_i)\in\mathbb{Z}_{\geq0}$,
if $y\alpha_i\in -\Delta_+$, $-(\beta,y\alpha_i)\in\mathbb{Z}_{>0}$,
where $\beta\in \mathfrak{h}^*, y\in W$.
Then from Theorem 8 of \cite{Kac77}, we have
$(\hat{\lambda}, \alpha_i^\vee)\in \mathbb{Z}_{\geq0} ~(i=1,\cdots,n-1,n+1,\cdots,n+m)$ and $(\hat{\lambda}, {\alpha_n^\prime}^\vee)\in \mathbb{Z}_{\geq0}$,
i.e.,
$$\begin{cases}
  \frac{2(n-m)+1}{2u}(\beta,\alpha_i)+(\rho,y^{-1}\alpha_i-\alpha_i)\in \frac{1}{2}\mathbb{Z}_{\geq0},&i=1,\cdots,n-1,\\
 \frac{2(n-m)+1}{2u}(\beta,\alpha_n^\prime)+(\rho,y^{-1}\alpha_n^\prime-\alpha_n^\prime)\in \mathbb{Z}_{\geq0},\\
  \frac{2(n-m)+1}{2u}(\beta,\alpha_i)+(\rho,y^{-1}\alpha_i-\alpha_i)\in -\frac{1}{2}\mathbb{Z}_{\geq0},&i=n+1,\cdots,n+m-1,\\
 \frac{2(n-m)+1}{2u}(\beta,\alpha_{n+m})+(\rho,y^{-1}\alpha_{n+m}-\alpha_{n+m})\in -\frac{1}{4}\mathbb{Z}_{\geq0}.\\
\end{cases}$$

If $y\notin W_1$, then from Corollary 10.2 of \cite{Hum78}, there exist $\alpha_i ~(i\in\{n+1,\cdots,n+m\})$ satisfying $y^{-1}\alpha_i\in -\Delta_{2,+}$,
thus $(\beta,\alpha_i)=-(\beta,y(-y^{-1}\alpha_i))\in\mathbb{Z}_{\geq0}$ and
$-\frac{2(n-m)+1}{2u}(\beta,\alpha_i)-(\rho,y^{-1}\alpha_i-\alpha_i)\geq0$.
Since $y^{-1}\alpha_i-\alpha_i\in-Q_{2,+}$, we have $(\rho,y^{-1}\alpha_i-\alpha_i)>0$,
this is impossible.
If $1\ne y\in W_1^\prime$, where $W_1^\prime=\langle r_{\alpha_i}|i=1,\cdots,n-1\rangle$, it is clear that $y^{-1}\theta\in \Delta_{1,+}$.
Note that $\theta$ is a long root of $sp_{2n}$, $\Delta_{1}^l=\{ \pm2\sum_{i=k}^{n-1} \alpha_i+\alpha_n^\prime ~(k=1, \dots, n-1), \pm\alpha_n^\prime\}$
is the set of positive long roots of $\Delta_1$
and $(y\rho, \alpha_j)<0$ for certain $j=1,\cdots,n-1$,
similar to Lemma \ref{lemy3}, there exists $\alpha\in Q_{1,+}$ satisfying $\theta+y^{-1}\alpha\in Q_+$ and $(\rho,\alpha-y^{-1}\alpha)>(n-m)+\frac{1}{2}$, where $Q_{1,+}=\bigoplus_{i=1}^{n-1}\mathbb{Z}_{\geq0}\alpha_i\oplus\mathbb{Z}_{\geq0}\alpha_n^\prime$.
Then we have $u-(\beta,\alpha)=u+(\beta,y\theta)-(\beta,y(\theta+y^{-1}\alpha))\in\mathbb{Z}_{\geq0}$ and $\frac{2(n-m)+1}{2u}(\beta,\alpha)+(\rho,y^{-1}\alpha-\alpha)\geq0$, thus $(\beta,\alpha)\geq\frac{2u}{2(n-m)+1}(\rho,\alpha-y^{-1}\alpha)>u$,
this is also impossible.
If $y\in W_1\backslash W_1^\prime$, from the proof of Theorem \ref{thmbmn},
there exists \(\alpha \in \Delta_{1,+}^l\) satisfying \(y^{-1}\alpha \in -\Delta_{1,+}^l\),
then we have $(\beta, \alpha)\in \mathbb{Z}_{>0}$,
$u-(\beta,\alpha)\in\mathbb{Z}_{\geq0}$,
and $(\beta,\alpha)\in \frac{2u}{2(n-m)+1}((\rho,\alpha-y^{-1}\alpha)+\mathbb{Z}_{\geq0})$, note that $(2u,2(n-m)+1)=1$, this is also impossible.
If $y=1$, we have
$$\begin{cases}
u+(\beta,\theta)\in\mathbb{Z}_{>0},\\
 (\beta,\alpha_i)\in-\mathbb{Z}_{\geq0},&i=1,\cdots,n+m,\\
(\beta,\alpha_i)\in \frac{u}{2(n-m)+1}\mathbb{Z}_{\geq0},&i=1,\cdots,n-1,\\
(\beta,\alpha_n^\prime)\in  \frac{2u}{2(n-m)+1}\mathbb{Z}_{\geq0},\\
(\beta,\alpha_i)\in -\frac{u}{2(n-m)+1}\mathbb{Z}_{\geq0},&i=n+1,\cdots,n+m-1,\\
(\beta,\alpha_{n+m})\in -\frac{u}{4(n-m)+2}\mathbb{Z}_{\geq0},
\end{cases}$$
since $(2u,2(n-m)+1)=1$, we have $(\beta,\alpha_i)=0$ for $i=1,\cdots,n+m$, thus $\beta=0$, i.e., $\hat{\lambda}=\mathcal{k}\Lambda_0$.
This shows that $L_{\widehat{\mathfrak{g}}}(\mathcal{k},0)$ itself is the only irreducible ordinary $L_{\widehat{\mathfrak{g}}}(\mathcal{k},0)$-module.
\end{proof}
\subsubsection{Type $D(m,n)$}

Let $\mathfrak{g}$ be the Lie superalgebra of type $D(m,n)$,
that is, $\mathfrak{g}=osp(2m|2n)$, where $m>1, n>0$.
We fix the Dynkin diagram $$
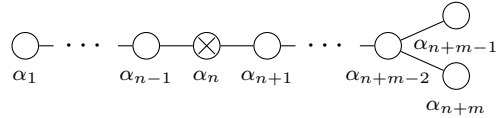
\begin{tikzpicture}[
    baseline=-0.5ex,
    line width=0.4pt,
    root/.style={circle, draw, minimum size=10pt, inner sep=0pt},
    otimes/.style={circle, draw, minimum size=10pt, inner sep=0pt,
        path picture={
            \draw (-3.1pt,-3.1pt) -- (3.1pt, 3.1pt);
            \draw (-3.1pt, 3.1pt) -- (3.1pt,-3.1pt);
        }
    }
]
    \node[root, label=below:{$\scriptstyle \alpha_1$}]      (n1) at (0, 0) {};
    \node (dots1) at (0.8, 0) {$\cdots$};
    \node[root, label=below:{$\scriptstyle \alpha_{n-1}$}]  (n2) at (1.6, 0) {};
    \node[otimes, label=below:{$\scriptstyle \alpha_n$}]    (n3) at (2.4, 0) {};
    \node[root, label=below:{$\scriptstyle \alpha_{n+1}$}]  (n4) at (3.2, 0) {};
    \node (dots2) at (4.0, 0) {$\cdots$};
    \node[root, label=below:{$\scriptstyle \alpha_{n+m-2}$}] (n5) at (4.8, 0) {};
    \node[root, label=below:{$\scriptstyle \alpha_{n+m-1}$}] (n6) at (5.7, 0.4) {};
    \node[root, label=below:{$\scriptstyle \alpha_{n+m}$}]   (n7) at (5.7, -0.4) {};
    \draw (n1) -- (dots1) -- (n2) -- (n3) -- (n4) -- (dots2) -- (n5);
    \draw (n5) -- (n6);
    \draw (n5) -- (n7);
\end{tikzpicture}
$$
and $\Pi=\{\alpha_1,\cdots,\alpha_{n+m}\}$,
where $\alpha_1,\cdots,\alpha_{n-1},\alpha_{n+1},\cdots,\alpha_{n+m}$ are even and $\alpha_n$ is odd.
Note that $\theta=2\alpha_1+\cdots+2\alpha_{n+m-2}+\alpha_{n+m-1}+\alpha_{n+m}$.
Let $\alpha_n^\prime=2\alpha_n+\cdots+2\alpha_{n+m-2}+\alpha_{n+m-1}+\alpha_{n+m}$.
We have $W=W_1\times W_2$, where $W_1=\langle r_{\alpha_i}, r_{\alpha_n^\prime}\mid i=1,\cdots,n-1\rangle$ and $W_2=\langle r_{\alpha_i}|i=n+1,\cdots,n+m\rangle$.
We also have $\Delta^0=\Delta_1\cup\Delta_2$,
and we can identify $\Delta_1$ with the root system of $sp_{2n}$,
identify $W_1$ with the Weyl group of $sp_{2n}$, identify $\theta$ with the highest root of $sp_{2n}$,
identify $\Delta_2$ with the root system of $o_{2m}$,
identify $W_2$ with the Weyl group of $o_{2m}$ and identify
$\theta^\prime=\alpha_{n+1}+2\alpha_{n+2}+\cdots+2\alpha_{n+m-2}+\alpha_{n+m-1}+\alpha_{n+m}$ with the highest root of $o_{2m}$.
From Lemma \ref{boundary}, all boundary admissible levels of $\widehat{\mathfrak{g}}$ are principal.
In the following, we classify the irreducible ordinary modules for $L_{\widehat{\mathfrak{g}}}(\mathcal{k},0)$
by considering two cases, one is $m>n$, the other is $m\leq n$.

{\bf Case 1: $m>n$.}
Note that $h^{\vee}=2(m-n-1)$, $r^\vee=1$,
$(\rho, \alpha_i)=-1$ for $i=1,\cdots,n-1$, $(\rho, \alpha_n)=0$,
$(\rho, \alpha_i)=1$ for $i=n+1,\cdots,n+m$.
If $m=n+1$, $h^{\vee}=0$, by Lemma \ref{boundary}, $\widehat{\mathfrak{g}}$ does not admit boundary admissible levels,
thus we assume that $m>n+1$.
Thus the main result for type $D(m,n)~(m>n+1)$ is the following theorem.
\begin{thm}
Let $\mathfrak{g}=osp(2m|2n)~(m>n+1>1)$ and
$\mathcal{k}=\frac{2(m-n-1)}{u}-2(m-n-1)$ be a boundary admissible level of $\widehat{\mathfrak{g}}$.
Then $L_{\widehat{\mathfrak{g}}}(\mathcal{k},0)$ itself is the only irreducible ordinary $L_{\widehat{\mathfrak{g}}}(\mathcal{k},0)$-module.
\end{thm}
\begin{proof}
Let $L(\hat{\lambda})$ be an irreducible ordinary $L_{\widehat{\mathfrak{g}}}(\mathcal{k},0)$-module,
from Lemma \ref{lemor} and Lemma \ref{lemboundary},
we obtain $\hat{\lambda}=(t_\beta y).(\mathcal{k}\Lambda_0)$
satisfying if $y\theta\in -\Delta_+$, $u+(\beta,y\theta)\in\mathbb{Z}_{\geq0}$,
if $y\theta\in \Delta_+$, $u+(\beta,y\theta)\in\mathbb{Z}_{>0}$,
and for any $i=1,\cdots,n+m$,
if $y\alpha_i\in \Delta_+$, $-(\beta,y\alpha_i)\in\mathbb{Z}_{\geq0}$,
if $y\alpha_i\in -\Delta_+$, $-(\beta,y\alpha_i)\in\mathbb{Z}_{>0}$,
where $\beta\in \mathfrak{h}^*, y\in W$.
Then from Theorem 8 of \cite{Kac77}, we have
$$\begin{cases}
  \frac{2(m-n-1)}{u}(\beta,\alpha_i)+(\rho,y^{-1}\alpha_i-\alpha_i)\in -\mathbb{Z}_{\geq0},&i=1,\cdots,n-1,\\
 \frac{2(m-n-1)}{u}(\beta,\alpha_n^\prime)+(\rho,y^{-1}\alpha_n^\prime-\alpha_n^\prime)\in -2\mathbb{Z}_{\geq0},\\
  \frac{2(m-n-1)}{u}(\beta,\alpha_i)+(\rho,y^{-1}\alpha_i-\alpha_i)\in \mathbb{Z}_{\geq0},&i=n+1,\cdots,n+m.
\end{cases}$$

If $y\notin W_1$ or $1\ne y\in W_1^\prime$,
where $W_1^\prime=\langle r_{\alpha_i}|i=1,\cdots,n-1\rangle$,
similar to the proof of Theorem \ref{thmbmn},
this is impossible;
if $y\in W_1\backslash W_1^\prime$,
note that $(u, 2(m-n-1))=1$, we have $(\beta,\alpha_i)\in \mathbb{Z}u$ for any $i=1,\cdots,n$
and $(\beta,\alpha_i)=0$ for any $i=n+1,\cdots,n+m$,
thus for any \(\alpha \in \Delta_{1,+}^l\), we have $(\beta,\alpha)\in 2u\mathbb{Z}$,
this is also impossible;
if $y=1$, then we have $\beta=0$, i.e., $\hat{\lambda}=\mathcal{k}\Lambda_0$.
This shows that $L_{\widehat{\mathfrak{g}}}(\mathcal{k},0)$ itself is the only irreducible ordinary $L_{\widehat{\mathfrak{g}}}(\mathcal{k},0)$-module.
\end{proof}

{\bf Case 2: $m\leq n$.}
Note that $h^{\vee}=n-m+1$, $r^\vee=2$,
$(\rho, \alpha_i)=\frac{1}{2}$ for $i=1,\cdots,n-1$, $(\rho, \alpha_n)=0$,
$(\rho, \alpha_i)=-\frac{1}{2}$ for $i=n+1,\cdots,n+m$.
Thus the main result for type $D(m,n)~(m\leq n)$ is the following theorem.
\begin{thm}
Let $\mathfrak{g}=osp(2m|2n)~(1<m\leq n)$ and
$\mathcal{k}=\frac{n-m+1}{u}-(n-m+1)$ be a boundary admissible level of $\widehat{\mathfrak{g}}$.
Then $L_{\widehat{\mathfrak{g}}}(\mathcal{k},0)$ itself is the only irreducible ordinary $L_{\widehat{\mathfrak{g}}}(\mathcal{k},0)$-module.
\end{thm}
\begin{proof}
Let $L(\hat{\lambda})$ be an irreducible ordinary $L_{\widehat{\mathfrak{g}}}(\mathcal{k},0)$-module,
from Lemma \ref{lemor} and Lemma \ref{lemboundary},
we obtain $\hat{\lambda}=(t_\beta y).(\mathcal{k}\Lambda_0)$
satisfying if $y\theta\in -\Delta_+$, $u+(\beta,y\theta)\in\mathbb{Z}_{\geq0}$,
if $y\theta\in \Delta_+$, $u+(\beta,y\theta)\in\mathbb{Z}_{>0}$,
and for any $i=1,\cdots,n+m$,
if $y\alpha_i\in \Delta_+$, $-(\beta,y\alpha_i)\in\mathbb{Z}_{\geq0}$,
if $y\alpha_i\in -\Delta_+$, $-(\beta,y\alpha_i)\in\mathbb{Z}_{>0}$,
where $\beta\in \mathfrak{h}^*, y\in W$.
Then from Theorem 8 of \cite{Kac77}, we have
$$\begin{cases}
  \frac{n-m+1}{u}(\beta,\alpha_i)+(\rho,y^{-1}\alpha_i-\alpha_i)\in \frac{1}{2}\mathbb{Z}_{\geq0},&i=1,\cdots,n-1,\\
 \frac{n-m+1}{u}(\beta,\alpha_n^\prime)+(\rho,y^{-1}\alpha_n^\prime-\alpha_n^\prime)\in \mathbb{Z}_{\geq0},\\
  \frac{n-m+1}{u}(\beta,\alpha_i)+(\rho,y^{-1}\alpha_i-\alpha_i)\in -\frac{1}{2}\mathbb{Z}_{\geq0},&i=n+1,\cdots,n+m.
\end{cases}$$

Similar to the proof of Theorem \ref{thmbmn1},
if $y\notin W_1$ or $1\ne y\in W_1^\prime$, where $W_1^\prime=\langle r_{\alpha_i}|i=1,\cdots,n-1\rangle$,
 this is impossible;
if $y=1$, we have $\beta=0$, i.e., $\hat{\lambda}=\mathcal{k}\Lambda_0$.
For $y\in W_1\backslash W_1^\prime$,
suppose $y\theta\in \Delta_+$, from the proof of Theorem \ref{thmbmn},
there exists \(\alpha \in \Delta_{1,+}^l\) satisfying \(y^{-1}\alpha \in -\Delta_{1,+}^l\), then we have $(\beta, \alpha)\in \mathbb{Z}_{>0}$,
$u-(\beta,\alpha)\in\mathbb{Z}_{>0}$,
and $(\beta,\alpha)\in \frac{u}{n-m+1}((\rho,\alpha-y^{-1}\alpha)+\mathbb{Z}_{\geq0})$. Note that $(u,n-m+1)=1$, thus this is impossible,
so we have $y\theta\in -\Delta_+$. 
Since $(u,2)=1$, we have $(\beta,\alpha_i)\in \mathbb{Z}u$ for any $i\in\{1,\cdots,n-1,n+1,\cdots,n+m\}$ and $(\beta,\alpha_n^\prime)\in \mathbb{Z}u$.
Let $\beta^\prime=-y^{-1}\beta$, then $(\beta^\prime, \alpha_i)=-(\beta,y\alpha_i)\in\mathbb{Z}_{\geq0}u$ for any $i\in\{1,\cdots,n-1,n+1,\cdots,n+m\}$
and $(\beta^\prime, \alpha_n^\prime)\in\mathbb{Z}_{\geq0}u$.
Since $(\beta^\prime, \theta)=-(\beta,y\theta)\leq u$,
we have $(\beta^\prime, \theta)=0$ or $(\beta^\prime, \theta)=u$.
Suppose $(\beta^\prime, \theta)=0$, then $(\beta^\prime, \alpha_i)=0$ for any $i\in\{1,\cdots,n-1\}$
and $(\beta^\prime, \alpha_n^\prime)=0$, but
note that there exists $\alpha\in\{\alpha_i,i=1,\cdots,n-1,\alpha_n^\prime\}$
satisfying $y\alpha\in -\Delta_+$, i.e., $-(\beta,y\alpha)\in\mathbb{Z}_{>0}$, this is impossible.
Thus $(\beta^\prime, \theta)=u$,
then $(\beta^\prime, \alpha_i)=0$ for any $i\in\{1,\cdots,n-1\}$
and $(\beta^\prime, \alpha_n^\prime)=u$.
Then we have $-(\beta,y\alpha_i)=0$ for any $i=1,\cdots,n-1$,
i.e., $y\alpha_i\in \Delta_+$ for any $i=1,\cdots,n-1$,
it implies $y\alpha_n^\prime\in -\Delta_+$.
Let $\alpha_n^\prime=2\alpha_n+R$, where $R=2\alpha_{n+1}+\cdots+2\alpha_{n+m-2}+\alpha_{n+m-1}+\alpha_{n+m}$,
then $y\alpha_n=\frac{1}{2}(y\alpha_n^\prime-yR)=\frac{1}{2}(y\alpha_n^\prime-R)\in -\Delta_+$,
thus we have $-(\beta,y\alpha_n)\in\mathbb{Z}_{>0}$, i.e., $(\beta^\prime, \alpha_n)\in\mathbb{Z}_{>0}$,
but $(\beta^\prime, \alpha_n)=\frac{1}{2}((\beta^\prime,\alpha_n^\prime)-(\beta^\prime,R))\in\frac{u}{2}(1-\mathbb{Z}_{\geq0})$ and $(u,2)=1$,
 this is also impossible.
 This implies that $L_{\widehat{\mathfrak{g}}}(\mathcal{k},0)$ itself is the only irreducible ordinary $L_{\widehat{\mathfrak{g}}}(\mathcal{k},0)$-module.
\end{proof}

\subsubsection{Type $F(4)$}
Let $\mathfrak{g}$ be the Lie superalgebra of type $F(4)$. We fix the Dynkin diagram
$$
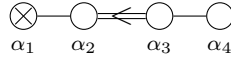
\begin{tikzpicture}[
    baseline=-0.5ex,
    line width=0.4pt,
    root/.style={circle, draw, minimum size=10pt, inner sep=0pt},
    otimes/.style={circle, draw, minimum size=10pt, inner sep=0pt,
        path picture={
            \draw (-3.1pt,-3.1pt) -- (3.1pt, 3.1pt);
            \draw (-3.1pt, 3.1pt) -- (3.1pt,-3.1pt);}},
   blackroot/.style={circle, draw, fill=black, minimum size=10pt, inner sep=0pt}]
    \node[ otimes, label=below:{$\scriptstyle \alpha_1$}]    (n1) at (0, 0) {};
     \node[root, label=below:{$\scriptstyle \alpha_{2}$}]    (n2) at (0.8, 0) {};
    \node[root, label=below:{$\scriptstyle \alpha_{3}$}] (n3) at (1.8, 0) {};
      \node[root, label=below:{$\scriptstyle \alpha_{4}$}]    (n4) at (2.6, 0) {};
    \draw (n1) -- (n2);  \draw(n3)--(n4);
    \draw[double, double distance=1.5pt] (n2) -- (n3) node[midway] {$<$};
\end{tikzpicture}
$$
and $\Pi=\{\alpha_1,\alpha_2,\alpha_3,\alpha_{4}\}$,
where $\alpha_2,\alpha_3,\alpha_{4}$ are even and $\alpha_1$ is odd.
Then the corresponding Cartan matrix is
$$A=\begin{pmatrix}
 0 & -\frac{1}{2} &0& 0\\
 -1 & 2 &-2& 0\\
 0 & -1 &2& -1\\
 0 & 0 &-1& 2
\end{pmatrix}
,$$
$\theta=2\alpha_1+3\alpha_2+2\alpha_{3}+\alpha_4$, and
$W=W_1\times W_2$, where $W_1=\langle r_{\theta}\rangle, W_2=\langle r_{\alpha_i}(i=2,3,4)\rangle$.
Note that $h^{\vee}=3$, $r^\vee=2$ and $(\rho, \alpha_1)=0,(\rho, \alpha_2)=\frac{1}{2}, (\rho, \alpha_3)=(\rho, \alpha_4)=1$.
We can view $W_1$ as the Weyl group of $sl_2$ and view $W_2$ as the Weyl group of $o_7$.
From Lemma \ref{boundary}, all boundary admissible levels of $\widehat{\mathfrak{g}}$ are principal, thus the main result for type $F(4)$ is the following theorem.

\begin{thm}\label{thmf4}
Let $\mathfrak{g}$ be the Lie superalgebra of type $F(4)$ and
$\mathcal{k}=\frac{3}{u}-3$ a boundary admissible level of $\widehat{\mathfrak{g}}$.
Then $L_{\widehat{\mathfrak{g}}}(\mathcal{k},0)$ itself is the only irreducible ordinary $L_{\widehat{\mathfrak{g}}}(\mathcal{k},0)$-module.
\end{thm}
\begin{proof}
Let $L(\hat{\lambda})$ be an irreducible ordinary $L_{\widehat{\mathfrak{g}}}(\mathcal{k},0)$-module,
from Lemma \ref{lemor} and Lemma \ref{lemboundary}, we have $\hat{\lambda}=(t_\beta y).(\mathcal{k}\Lambda_0)$
satisfying if $y\theta\in -\Delta_+$, $u+(\beta,y\theta)\in\mathbb{Z}_{\geq0}$,
if $y\theta\in \Delta_+$, $u+(\beta,y\theta)\in\mathbb{Z}_{>0}$,
and for any $i=1,\cdots,4$,
if $y\alpha_i\in \Delta_+$, $-(\beta,y\alpha_i)\in\mathbb{Z}_{\geq0}$,
if $y\alpha_i\in -\Delta_+$, $-(\beta,y\alpha_i)\in\mathbb{Z}_{>0}$,
where $\beta\in \mathfrak{h}^*, y\in W$.
Then from Theorem 8 of \cite{Kac77}, we have
$(\hat{\lambda}, \alpha_i^\vee)\in \mathbb{Z}_{\geq0} ~(i=2,3,4)$ and $(\hat{\lambda}, \theta^\vee)\in \mathbb{Z}_{\geq0}$,
i.e.,
$$\begin{cases}
 \frac{3}{u}(\beta,\alpha_2)+(\rho,y^{-1}\alpha_2-\alpha_2)\in \frac{1}{2}\mathbb{Z}_{\geq0},\\
  \frac{3}{u}(\beta,\alpha_i)+(\rho,y^{-1}\alpha_i-\alpha_i)\in \mathbb{Z}_{\geq0}, &i=3,4,\\
 \frac{3}{u}(\beta,\theta)+(\rho,y^{-1}\theta-\theta)\in -\frac{3}{2}\mathbb{Z}_{\geq0}.
\end{cases}$$

If $y\notin W_1$, similar to Lemma \ref{lemy3}, there exists $\alpha\in Q_{2,+}$ satisfying $\theta+y^{-1}\alpha\in Q_+$ and $(\rho,\alpha-y^{-1}\alpha)>3$, where $Q_{2,+}=\bigoplus_{i=2}^{4}\mathbb{Z}_{\geq0}\alpha_i$.
Then we have $u-(\beta,\alpha)=u+(\beta,y\theta)-(\beta,y(\theta+y^{-1}\alpha))\in\mathbb{Z}_{\geq0}$ and $\frac{3}{u}(\beta,\alpha)+(\rho,y^{-1}\alpha-\alpha)\geq0$, thus $(\beta,\alpha)\geq\frac{u}{3}(\rho,\alpha-y^{-1}\alpha)>u$,
this is impossible.
If $y=r_\theta$, we have
$$\begin{cases}
u-(\beta,\theta)\in\mathbb{Z}_{\geq0},\\
 (\beta,\alpha_1+3\alpha_2+2\alpha_{3}+\alpha_4)\in\mathbb{Z}_{>0},\\
 -(\beta,\alpha_i)\in\mathbb{Z}_{\geq0},&i=2,3,4,\\
 \frac{3}{u}(\beta,\alpha_2)\in \frac{1}{2}\mathbb{Z}_{\geq0},\\
  \frac{3}{u}(\beta,\alpha_i)\in \mathbb{Z}_{\geq0},&i=3,4,\\
 (\beta,\theta)\in 3u-\frac{u}{2}\mathbb{Z}_{\geq0},
\end{cases}$$
note that $(u,2)=(u,3)=1$, this is impossible.
If $y=1$, then we have
$$\begin{cases}
u+(\beta,\theta)\in\mathbb{Z}_{>0},\\
 -(\beta,\alpha_i)\in\mathbb{Z}_{\geq0},&i=1,2,3,4,\\
 \frac{3}{u}(\beta,\alpha_2)\in \frac{1}{2}\mathbb{Z}_{\geq0},\\
  \frac{3}{u}(\beta,\alpha_i)\in \mathbb{Z}_{\geq0},&i=3,4,\\
 (\beta,\theta)\in -\frac{u}{2}\mathbb{Z}_{\geq0},
\end{cases}$$
note that $(u,2)=(u,3)=1$, we have $(\beta,\alpha_i)=0$ for $i=1,\cdots,4$, thus $\beta=0$, i.e., $\hat{\lambda}=\mathcal{k}\Lambda_0$.
This implies that $L_{\widehat{\mathfrak{g}}}(\mathcal{k},0)$ itself is the only irreducible ordinary $L_{\widehat{\mathfrak{g}}}(\mathcal{k},0)$-module.
\end{proof}

\subsubsection{Type $G(3)$}
Let $\mathfrak{g}$ be the Lie superalgebra of type $G(3)$.
We fix the Dynkin diagram
$$
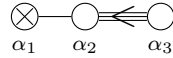
\begin{tikzpicture}[
    baseline=-0.5ex,
    line width=0.4pt,
    root/.style={circle, draw, minimum size=10pt, inner sep=0pt},
    otimes/.style={circle, draw, minimum size=10pt, inner sep=0pt,
        path picture={
            \draw (-3.1pt,-3.1pt) -- (3.1pt, 3.1pt);
            \draw (-3.1pt, 3.1pt) -- (3.1pt,-3.1pt);}},
   blackroot/.style={circle, draw, fill=black, minimum size=10pt, inner sep=0pt}]
    \node[ otimes, label=below:{$\scriptstyle \alpha_1$}]    (n1) at (0, 0) {};
     \node[root, label=below:{$\scriptstyle \alpha_{2}$}]    (n2) at (0.8, 0) {};
    \node[root, label=below:{$\scriptstyle \alpha_{3}$}] (n3) at (1.8, 0) {};
    \draw (n1) -- (n2);
    \draw[double, double distance=2.5pt] (n2) -- (n3);
    \draw (n2) -- (n3) node[midway] {\large $<$};
\end{tikzpicture}
$$
and $\Pi=\{\alpha_1,\alpha_2,\alpha_3\}$,
where $\alpha_2,\alpha_3$ are even and $\alpha_1$ is odd.
Then the corresponding Cartan matrix is
$$A=\begin{pmatrix}
 0 & -\frac{1}{3} &0\\
 -1 & 2 &-3\\
 0 & -1 &2
\end{pmatrix}
,$$
and $\theta=2\alpha_1+4\alpha_2+2\alpha_{3}$,
$W=W_1\times W_2$, where $W_1=\langle r_{\theta}\rangle, W_2=\langle r_{\alpha_2},r_{\alpha_3}\rangle$.
Note that $h^{\vee}=2$, $r^\vee=3$ and $(\rho, \alpha_1)=0,(\rho, \alpha_2)=\frac{1}{3}, (\rho, \alpha_3)=1$.
We can view $W_1$ as the Weyl group of $sl_2$ and view $W_2$ as the Weyl group of Lie algebra of type $G_2$.
From Lemma \ref{boundary}, all boundary admissible levels of $\widehat{\mathfrak{g}}$ are principal.
Similar to the proof of Theorem \ref{thmf4}, the main result for type $G(3)$ is the following theorem.

\begin{thm}
Let $\mathfrak{g}$ be the Lie superalgebra of type $G(3)$ and
$\mathcal{k}=\frac{2}{u}-2$ a boundary admissible level of $\widehat{\mathfrak{g}}$.
Then $L_{\widehat{\mathfrak{g}}}(\mathcal{k},0)$ itself is the only irreducible ordinary $L_{\widehat{\mathfrak{g}}}(\mathcal{k},0)$-module.
\end{thm}

\end{document}